\documentclass[final]{siamltex}
\usepackage{xcolor}
\usepackage{amsmath}
\usepackage{amssymb}
\usepackage{mathabx}
\usepackage{mathrsfs}
\usepackage{graphicx}
\usepackage{bm}
\usepackage{ucs}
\usepackage[utf8x]{inputenc}
\usepackage{wrapfig}
\usepackage{color}
\usepackage{float}
 \usepackage{algorithm,algorithmic}
\usepackage[breaklinks,colorlinks=true,linkcolor=blue,citecolor=red, backref=page]{hyperref}

\newcommand{\cF}{\hat{\mathcal{F}}}

\newcommand{\rc}{\textcolor{black}}

\renewcommand{\tilde}{\widetilde}

\DeclareMathAlphabet{\itbf}{OML}{cmm}{b}{it}

\DeclareMathAlphabet{\itbf}{OML}{cmm}{b}{it}
 \DeclareMathAlphabet\mathbfcal{OMS}{cmsy}{b}{n}

\newtheorem{prop}[theorem]{Proposition}
\renewcommand{\hat}{\widehat}
\def\EE{\mathbb{E}}
\def\CC{\mathbb{C}}
\def\RR{\mathbb{R}}
\def\eps{\varepsilon}
\def\om{\omega}

\def\bkappa{{\boldsymbol{\kappa}}}
\def\bzeta{{\boldsymbol{\zeta}}}
\def\bxi{{\boldsymbol{\xi}}}
\newcommand{\vka}{\vec{\bkappa}}

\def\bx{{\itbf x}}
\def\vx{\vec{\bx}}
\def\by{{\itbf y}}

\def\bu{\boldsymbol{\eta}}

\def\bq{{\itbf q}}

\def\bX{{\itbf X}}
\def\obX{{\itbf X}}
\def\tbX{{\itbf Y}}

\def\tbq{\tilde{\bq}}

\def\cC{\mathcal{C}}
\def\cW{\mathcal{W}}

\begin{document}

\title{Paraxial wave propagation in random media with long-range correlations}

\author{
Liliana Borcea\footnotemark[1]
\and
Josselin Garnier\footnotemark[2]
\and 
Knut S{\O}lna\footnotemark[3]
}

\maketitle

\renewcommand{\thefootnote}{\fnsymbol{footnote}}

\footnotetext[1]{Department of Mathematics, University of Michigan,
  Ann Arbor, MI 48109. {\tt borcea@umich.edu}}
\footnotetext[2]{Centre de Math\'ematiques Appliqu\'ees, Ecole
  Polytechnique, Institut Polytechnique de Paris, 91128 Palaiseau Cedex, France.  {\tt
    josselin.garnier@polytechnique.edu}}
\footnotetext[3]{Department of Mathematics,
University of California at Irvine,
Irvine, CA 92697. {\tt ksolna@math.uci.edu}}

\begin{abstract}
We study the paraxial wave equation with a randomly perturbed index of refraction, which can model the  propagation of a wave beam in a turbulent medium. The random perturbation is  a  stationary and isotropic  process with a general form of the covariance that may be integrable or not.  We focus attention mostly on the non-integrable  case, which corresponds to a random perturbation with long-range correlations, that is relevant for propagation through a cloudy turbulent atmosphere. The analysis is carried out in a high-frequency regime where the forward scattering approximation holds. It reveals that the 
randomization of the wave field is multiscale: The travel  time of the wave front is randomized at short distances of propagation and it can be described by a fractional Brownian motion. The wave field observed  in the  random travel time frame is affected by the random perturbations at long distances,  and it is described by a  Schr\"{o}dinger-type equation driven by a standard Brownian field. We use these results to quantify how scattering leads to decorrelation of the spatial and spectral components of the wave field  and to a deformation of the pulse emitted by the source. These are important questions for applications like imaging and free space communications with pulsed laser beams  through a turbulent atmosphere.  We also  compare the results with those used in the optics literature, which are based on  the Kolmogorov model of turbulence.

 \end{abstract}

\begin{keywords}
Paraxial wave equation, turbulent atmosphere,  asymptotic analysis, long-range correlations.
\end{keywords}

\begin{AMS}
76B15, 35Q94, 60F05.
\end{AMS}

\section{Introduction}
\label{sect:Intro}
The paraxial wave equation describes wave propagation along a privileged axis, as a narrow angle beam, in a homogeneous or heterogeneous medium
\cite{bamberger1988parabolic}. It is a  parabolic approximation of the wave equation, which neglects backscattering and thus facilitates  the analysis and computation of waves at long distance of propagation, aka range. The parabolic  approximation theory was introduced by Leontovich and Fock \cite{leontovich1946solution} and has been used and developed further in applied fields like seismology \cite{claerbout1970coarse,claerbout1976fundamentals}, underwater acoustics \cite{tappert1977parabolic}, optics \cite{hasegawa1973transmission} and laser optics \cite{andrews2005laser,ishimaru,tatarskii,young1996two}.  

Motivated by laser optics applications to imaging and free space  communications through a turbulent atmosphere, we consider the paraxial wave equation with a randomly perturbed wave speed $c(\vx)$. The model of the perturbation  is 
\begin{equation}
\frac{c_o^2}{c^2(\vx)} = 1 + 
\mu(\vx),
\label{eq:In1}
\end{equation}
where $c_o$ is the constant reference speed
and $\mu$ is a zero-mean, stationary and isotropic random process, with power spectral density (Fourier transform of the covariance)  of the form
\begin{equation}
{\mathbb{S}}(\vka) =  \int_{\RR^3} d \vx \, \EE \left[ \mu(\vx') \mu(\vx' + \vx)\right] e^{- i \vka \cdot \vx} = 
\chi_\alpha {\bf 1}_{(L_o^{-1},l_o^{-1})}(|\vka|) |\vec\bkappa|^{-2-\alpha}.
\label{def:psdS}
\end{equation}
Here $\chi_\alpha$ is a constant \rc{(expressed in unit of length to the power $1-\alpha$)},  $\alpha \in (0,1) \cup (1,2)$ and ${\bf 1}_{(L_o^{-1},l_o^{-1})}$ is the indicator function equal to one when  its argument is in $ (L_o^{-1},l_o^{-1})$ and 
zero otherwise. 

Definition \eqref{def:psdS} is a generalization of the commonly used Kolmogorov power spectrum, where $\alpha = 5/3$ and the ``outer scale" $L_o$ 
and the ``inner scale" $l_o$ define the ``inertial range" of turbulence \cite{andrews2005laser}. There is a growing number of  studies in the optics literature concerned with quantifying the effect of non-Kolmogorov turbulence on beam propagation \cite{charnotskii2012intensity,korotkova2021non,zilberman2008propagation}. All of them consider 
$\alpha >1$, which corresponds to an 
integrable covariance of $\mu$. This case is well understood from the mathematical point of view and has been 
analyzed in detail in the high-frequency, paraxial regime 
 in  \cite{fannjiang2004scaling,garnier2009coupled}. 
 The  wave field is described asymptotically by the solution of an It\^{o}-Schr\"{o}dinger equation driven by a Brownian field with covariance defined in terms of  $\mathbb{S}(\vka)$. Therefore, the second and even fourth order statistical moments of the wave field can be calculated using It\^{o} calculus \cite{garnier2016fourth}. The study of such moments is an 
essential part of  both the analysis and the development of new methodologies  
for imaging \cite{borcea2021imaging,borcea2007asymptotics,garnier2018imaging}, time reversal  \cite{blomgren2002super,fannjiang05atmos,gs22,papanicolaou2007self} and optical communications applications \cite{borcea2020multimode}.

The case $\alpha \in (0,1)$ has not been explored in the optics literature and it is interesting mathematically because depending on the outer scale $L_o$,
it may give a non-integrable covariance of the fluctuations, meaning that $\mu$ has long-range correlations. Moreover, $\alpha  <1$  is relevant for propagation through a cloudy atmosphere, as seen from the experimental studies \cite{cahalan1989marine} and \cite[Table 3]{madhavan2017multiresolution}. The conclusion of these studies is that the value of $\alpha$ depends on the interval $(L_o^{-1},l_o^{-1})$, with $\alpha < 1$ at length scales that are larger than the outer scale of  
Kolmogorov turbulence. Thus, one could consider an even more general model of the power spectrum, with $\alpha < 1$ at longer scales and 
$\alpha > 1$ at smaller scales. For brevity, we work with the model \eqref{def:psdS}, which is sufficient to display the effects of long range medium fluctuations  on the statistics 
of the wave beam. 

{Most of our analysis is concerned with  $\alpha \in (0,1)$ and a beam with initial radius of order $r_{\rm s}$, satisfying  $l_o \lesssim r_{\rm s} \ll L_o$,  so we can take  $L_o \to \infty$, while keeping $l_o$ finite. The covariance of $\mu$ is non-integrable in this case, which means that the classic paraxial theory in \cite{fannjiang2004scaling,garnier2009coupled} does not apply. 
We refer   to \cite{gomez2017fractional} for the derivation of the  paraxial approximation in a random anisotropic medium with 
 long-range correlation properties. There, the  wave  is described asymptotically by the solution of a  Schr\"{o}dinger equation with fractional white noise potential. In this paper we show that for our isotropic random  medium modeled by $\mu$, a  transformation involving the central axis  travel time (i.e., the travel time measured at the center of the beam), can convert  the problem to one where the classic analytic framework applies. We prove that there are two distinguished range scales that 
describe the net scattering effects  on the beam: The central axis travel time randomizes  on a small range scale and it is described by a fractional Brownian motion. }This behavior was also shown in \cite{bal2011asymptotics,solna03fract}. The shape  of the wave, observed in the random travel time frame, is not 
affected by scattering at this short range. However, this, too, randomizes at long range and it is described by the solution of an   It\^{o}-Schr\"{o}dinger equation driven by a standard Brownian field, like in \cite{fannjiang2004scaling,garnier2009coupled}. 
We use these asymptotic results to analyze explicitly the 
spatial and frequency covariance of the wave field.  This allows us to quantify how the wave components decorrelate 
and how the pulse emitted by the source  deforms  due to scattering in the random medium.

To relate our results with the existing optics literature, we also consider briefly the case 
$\alpha \in (0,1) \cup (1,2)$ with a finite $L_o$. 
These cases correspond to an integrable covariance of the process $\mu$, where the theory in \cite{fannjiang2004scaling,garnier2009coupled} applies. We study  the covariance  of the wave field, which  depends on $\alpha$ and the scales $l_o$ and $L_o$,  and 
 quantify explicitly the accuracy of the  approximations commonly used in optics literature \cite{andrews2005laser}.

The paper is organized as follows: We begin in section \ref{sect:Formulation} with the mathematical formulation of the problem. We state the paraxial 
wave equation, identify the asymptotic regime and give more details on the random process $\mu$. The asymptotic analysis 
for the case $\alpha \in (0,1)$, with $L_o \to \infty$ and finite $l_o$ is given in section \ref{sect:Asymptotic}. We use it in section \ref{sect:Application} to 
quantify the decorrelation of the wave components and the deformation of the pulse due to scattering. The comparison 
with the formulas in the optics literature are in section \ref{sect:Comparison}. We end with a summary in section \ref{sect:Summary}.

\section{Mathematical formulation} 
\label{sect:Formulation}
Let us introduce the orthogonal system of coordinates $\vx = (\bx,z)$, with range axis $z$ along the direction of propagation and with $\bx \in \RR^2$ 
in the cross-range plane. The wave field $u$ satisfies the wave equation
\begin{align}
\left[\frac{1}{c^2(\bx,z)}\partial_t^2 - \Delta_\bx - \partial_z^2 \right] u(t,\bx,z)  &= 
\partial_t \big[ 2 \cos(\om_o t)  f(Bt) \big] S\Big(\frac{\bx}{r_{\rm s}}\Big) \delta (z), \label{eq:WEq} 
\end{align}
for $(t,\bx,z) \in \RR \times \RR^2 \times \RR$, where $\Delta_\bx$ denotes the 
Laplacian with respect to $\bx$. The source  is localized at the origin of range and has a  cross-range profile  with radius $r_{\rm s}$, modeled by the function $S$ of 
dimensionless argument, with support centered at ${\bf 0}$.  The source signal is a pulse with bandwidth $B$, modulated at the carrier (center) frequency $\om_o$ and with envelope modeled by the   function $f$ of dimensionless argument.  Prior to the source excitation 
there is no wave:   
$u(t,\bx,z) \equiv 0,$ for $ t \ll -1/B.$

 Since the analysis 
of wave propagation requires the decomposition of the wave field over frequencies, we work henceforth in the Fourier domain,  
\begin{equation}
\hat u(\om,\bx,z) = \int_{-\infty}^\infty d t \,  e^{i \om t} u(t,\bx,z).
\label{eq:WFT}
\end{equation}
This time-harmonic wave satisfies the Helmholtz equation
\begin{align}
&
\left[\frac{\om^2}{c^2(\bx,z)} +\Delta_\bx +\partial_z^2 \right] \hat u(\om,\bx,z)  = 
i\omega 
\hat F(\om,\bx) \delta (z)  , \label{eq:WFT1} 
\end{align}
for $(\om,\bx,z) \in \RR \times \RR^2\times \RR$, with 
\begin{align}
\hat F(\om,\bx) = \frac{1}{B}
\Big[  \hat f\Big(\frac{\om-\om_o}{B}\Big) +  \hat f\Big(\frac{\om+\om_o}{B}\Big) \Big]
S\Big(\frac{\bx}{r_{\rm s}}\Big),
\label{def:hatF}
\end{align}
and  outgoing boundary conditions at $|(\bx,z)| \to \infty$. These conditions can be justified mathematically by truncating the random medium 
outside a ball of large enough radius, so that in the time domain, the truncation does not affect the wave over the duration  of interest.

We state next, in subsection \ref{sect:Scaling},  the paraxial approximation of  equation \eqref{eq:WFT1} and the asymptotic regime where it is valid. The details on the random process $\mu$ are in subsection \ref{sect:randmodel}. 

\subsection{Scaling and the paraxial equation}
\label{sect:Scaling}

The paraxial approximation holds  in a high-frequency regime, where the wavelength 
is much smaller than the radius of the beam and the correlation radius of the medium, which are,  in turn, much smaller than the range scale (distance of propagation).

We introduce the small dimensionless parameter $\eps>0$ that encapsulates this regime and assume that, compared to the typical range, the typical wavelength is of order $\eps^4$ and the beam radius and the correlation radius are of order $\eps^2$:
\begin{equation}
B^\eps = \frac{B}{\eps^4},\quad
\om_o^\eps = \frac{\om_o}{\eps^4},\quad
r_{\rm s}^\eps = \eps^2 r_{\rm s},\quad
l_o^\eps = \eps^2 l_o,\quad 
L_o^\eps = \eps^2 L_o,\quad 
\chi_\alpha^\eps = \chi_\alpha \eps^{8-2\alpha} .
\label{eq:scaling}
\end{equation}
 As we will see, the scaling of $\chi_\alpha^\eps$ is the one that gives a non-trivial limit as $\eps \to 0$. It is also possible to consider a larger range scale $L_o^\eps =\eps^p L_o $, with $p<2$, and/or a smaller $l_o^\eps = \eps^q l_o$, with $q>2$ 
\cite{fannjiang2004scaling}.
Here we  consider the scaling (\ref{eq:scaling}) and study the subsequent limits $L_o \to+\infty$ and/or $l_o \to 0$.

We denote by $\mu^\eps$ a random process with the power spectral density of the form (\ref{def:psdS}) with the constant $\chi_\alpha^\eps$ and scales $l_o^\eps$, $L_o^\eps$.
Then,   (\ref{eq:scaling})
gives the representation
\begin{equation}
 \label{eq:F2}
\mu^\eps(\vx) =  \eps^3 \mu \Big(\frac{\vx}{\eps^2}\Big),
\end{equation}
where $\mu$ is a random process with the power spectral density of the form (\ref{def:psdS}) with the constant $\chi_\alpha$ and scales $l_o$, $L_o$.
{The wave field in the scaling \eqref{eq:scaling} is denoted by $\hat u^\eps$ and it satisfies the following Helmholtz equation derived from (\ref{eq:WFT1})}
\begin{align}
&\left[\frac{\om^2}{c_o^2}[1+\mu^\eps (\bx,z)]+\Delta_\bx +\partial_z^2 \right] \hat u^\eps(\om,\bx,z)  =
 i \omega
 \hat F^\eps(\om,\bx) \delta (z)  ,
\label{eq:WFT1sca}
\end{align}
with
\begin{align}
& \hat F^\eps (\om,\bx) =\frac{1}{B^\eps} \Big[ \hat f\Big(\frac{\om-\om_o^\eps}{B^\eps}\Big) +  \hat f\Big(\frac{\om+\om_o^\eps}{B^\eps}\Big) \Big]
 S\Big(\frac{\bx}{r_{\rm s}^\eps}\Big)
 = \eps^4 \hat F\Big(\eps^4 \om , \frac{\bx}{\eps^2}\Big),
 \label{eq:WFT1scaF}
\end{align}
{and outgoing boundary conditions at $|(\bx,z)| \to \infty$.}

{Observe that if we had $S\equiv 1$ and $\mu \equiv 0$ in (\ref{eq:WFT1sca}-\ref{eq:WFT1scaF}), the solution would be the plane wave
$$
\hat u^\eps\big(  \om , \bx, z \big) = 
 \frac{c_o \eps^4}{2}\exp\Big(i \frac{\om}{c_o} z \Big) \frac{1}{B} \Big[ \hat f\Big(\frac{\eps^4 \om-\om_o}{B}\Big) +  \hat f\Big(\frac{\eps^4 \om+\om_o}{B}\Big) \Big].
$$
This observation motivates the introduction of the ``slowly varying envelope field" $\varphi^\eps$, 
which defines the solution of (\ref{eq:WFT1sca}-\ref{eq:WFT1scaF}) as follows
\begin{equation}
\label{eq:sve}
\hat u^\eps\big( \om  ,\bx,z\big)
=\
\frac{c_o \eps^4}{2}
\exp \Big( i \frac{\om}{c_o} z \Big) \varphi^\eps \Big(\eps^4 \om, \frac{\bx}{\eps^2}, z\Big).
\end{equation}
Substituting (\ref{eq:sve})  into (\ref{eq:WFT1sca}),
using the chain rule
and denoting $k(\varOmega)=\varOmega/c_o$, 
we find that for $\varOmega=\eps^4 \om \in \RR$ and $\bX=\bx/\eps^2\in \RR^2$, we have 
\begin{align}
\left[ 2 i k(\varOmega) \partial_z + \Delta_{\bX} + \frac{k^2(\varOmega)}{\eps} \mu \Big(\bX ,\frac{z}{\eps^2} \Big) \right]  \varphi^\eps (\varOmega, \bX ,z) &= 0, \quad z > 0,  
 \label{eq:Parax}\\
\varphi^\eps(\varOmega,\bX,z=0)  &=   \hat F (\varOmega, \bX ).
 \label{eq:ParaxIC}
\end{align}
In equation (\ref{eq:Parax}) we have neglected the  $\eps^4 \partial_z^2 \varphi^\eps$ term, which is responsible for  backscattering. 
Thus, we use the forward scattering approximation, which  can be justified when $\eps \to 0$  as shown in \cite{garnier2009coupled}.
}

\subsection{Statistics of the random fluctuations}
\label{sect:randmodel}
The most convenient choice for the analysis would be having a Gaussian $\mu$. However, since Gaussian processes are unbounded,  this choice is inconsistent with equation \eqref{eq:In1}, whose right hand side must be positive. We assume instead that $\mu$ is defined by a smooth and bounded function of a Gaussian process, which averages to zero  so  that $\EE[\mu] = 0$.  This gives a consistent random perturbation  model, while keeping the analysis simple enough.

The covariance of  $\mu$  is the inverse Fourier transform of the power spectrum  \eqref{def:psdS}
\begin{align}
\mbox{Cov}_\mu(\bX,z) = &\EE\big[ \mu(\bX',z') \mu(\bX'+\bX,z'+z) \big] = \frac{1}{(2 \pi)^3} \int_{\RR^3} d \vka \, \cos \big[\vka \cdot (\bX,z)\big] \mathbb{S}(\vka) \nonumber \\
&\qquad = \frac{\chi_\alpha}{(2 \pi)^3} \int_{L_o^{-1}}^{l_o^{-1}} d \kappa \, \kappa^{2} \int_0^{2 \pi} d \varphi \int_{0}^\pi d \vartheta \sin \vartheta\,  \kappa^{-2-\alpha} 
\cos[\kappa |(\bX,z)| \cos \vartheta] \nonumber \\
&\qquad = \frac{\chi_\alpha}{2 \pi^2} \int_{L_o^{-1}}^{l_o^{-1}} d \kappa \, \kappa^{-\alpha} \mbox{sinc} [ \kappa |(\bX,z)| ] \nonumber \\
&\qquad = \frac{\chi_\alpha |(\bX,z)|^{\alpha-1}}{2 \pi^2} \int_{|(\bX,z)|/L_o}^{|(\bX,z)|/l_o} ds \, s^{-\alpha} \mbox{sinc}(s). \label{eq:ST1}
\end{align}
Here we introduced the spherical coordinates $\vka \mapsto (\kappa, \varphi, \vartheta)$, with $\kappa = |\vka|$ and angles $\varphi \in (0, 2 \pi)$ and 
$\vartheta \in (0,\pi)$. We also changed the variable of integration to $s = \kappa |(\bX,z)|$. 
The variance of $\mu$ is obtained from equation \eqref{eq:ST1} evaluated at the origin,
\begin{align}
\mbox{Var}_\mu = \EE\big[ \mu^2(\bX,z) \big] &=  \frac{\chi_\alpha}{2 \pi^2} \int_{L_o^{-1}}^{l_o^{-1}} d \kappa \, \kappa^{-\alpha}  = \frac{\chi_\alpha}{2 \pi^2} \left(\frac{L_o^{\alpha-1}-l_o^{\alpha-1}}{\alpha-1} \right).
\label{eq:ST2}
\end{align}

We distinguish the following  two cases   in the paper. The first one is used in the analysis in sections \ref{sect:Asymptotic} and \ref{sect:Application}, while the other one is  used for comparison with the optics literature in section \ref{sect:Comparison}.

\vspace{0.1in}
$\bullet$ \textbf{$\alpha \in (0,1)$ and infinite outer scale:}
When  the initial radius $r_{\rm s}$ of the beam  satisfies the order relation  $l_o \lesssim r_{\rm s} \ll L_o$, we can carry out the analysis in the  limit  $L_o \to \infty$, while keeping $l_o$ finite.  The variance \eqref{eq:ST2} is finite in this limit
\begin{equation}
\mbox{Var}_\mu =  \frac{\chi_\alpha}{2 \pi^2(1-\alpha) l_o^{1-\alpha}}, \quad \alpha \in (0,1), ~~ L_o \to \infty, \label{eq:ST3}
\end{equation}
but the covariance \eqref{eq:ST1} is not integrable. In particular, we obtain from \eqref{eq:ST1} that 
\begin{align}
\mbox{Cov}_\mu({\bf 0},z)&= \frac{\chi_\alpha |z|^{\alpha-1}}{2 \pi^2} \int_{0}^{|z|/l_o} ds \, s^{-\alpha} \mbox{sinc}(s) \sim  \frac{C_\alpha}{2 \pi^2} |z|^{\alpha-1}, \quad \mbox{as} ~ |z| \to \infty,
\label{eq:ST4}
\end{align}
where the symbol ``$\sim$" denotes an asymptotic expansion and, according to \cite[Formula 3.761.4]{grad}, 
\begin{equation}
C_\alpha = \chi_\alpha \int_0^\infty ds \, s^{-\alpha} \mbox{sinc}(s) = \frac{\pi \chi_\alpha}{2 \cos(\alpha \pi/2) \Gamma(1+\alpha)}.
\label{eq:ST5}
\end{equation}
The slow decay at $|z| \to \infty$  in \eqref{eq:ST4} implies that $\mbox{Cov}_\mu$  is non-integrable and we say that   the process $\mu$ has long-range 
correlations. 

\vspace{0.1in}
$\bullet$ \textbf{$\alpha \in (0,1) \cup (1,2)$ and a finite outer scale:} When the beam has a larger radius, meaning that $l_o \lesssim r_{\rm s} \lesssim L_o$,
it experiences the random fluctuations in a different way than above,  even for $\alpha < 1$. Indeed, integration by parts gives the estimate 
\begin{align*}
\left|\int_{|z|/L_o}^\infty ds \, s^{-\alpha} \mbox{sinc}(s)\right| &= \left| \Big(\frac{|z|}{L_o}\Big)^{-\alpha - 1} \cos\Big(\frac{|z|}{L_o}\Big) - 
(1+\alpha) \int_{|z|/L_o}^\infty ds \, s^{-\alpha-2} \cos(s)\right| \\& \le \Big(\frac{|z|}{L_o}\Big)^{-\alpha - 1}  + (1+\alpha)  \int_{|z|/L_o}^\infty ds \, s^{-\alpha-2}  = 2 \Big(\frac{|z|}{L_o}\Big)^{-\alpha - 1} ,
\end{align*}
and substituting into \eqref{eq:ST1} evaluated at $(\bX,z) = ({\bf 0},z)$ we get 
\begin{align}
\mbox{Cov}_\mu({\bf 0},z) & \le \frac{\chi_\alpha L_o^{\alpha + 1}}{\pi^2} |z|^{-2}, \quad \mbox{as} ~ |z| \to \infty.
\label{eq:ST6}
\end{align}
The decay at $|z| \to \infty$ is now fast enough to make  the covariance  integrable and we say that the process $\mu$ is mixing.

Note from \eqref{eq:ST2} that when $\alpha \in (1,2)$, the variance of $\mu$ is finite only for a finite 
outer scale $L_o$, while the inner scale can be either finite or tend to $0$. For the case $\alpha \in (0,1)$ the variance blows up 
in the limit $l_o \to 0$, but it is finite for $L_o \to \infty$.

\section{Asymptotic analysis for the long-range correlation case}
\label{sect:Asymptotic}
We now describe  the  solution $\varphi^\eps$ of the paraxial equation  (\ref{eq:Parax}-\ref{eq:ParaxIC}) in the asymptotic limit $\eps \to 0$, for  $\alpha \in (0,1)$ and an infinite outer scale. This case is interesting because the process $\mu$ has long-range correlations and there are two range scales that describe the randomization of $\varphi^\eps$. We show in subsection \ref{sect:TravelT} that $\varphi^\eps$ develops a significant random phase at a short, $\eps$ dependent range scale.  Thus, in order to analyze  it  at longer range, we need to remove this random phase i.e., observe $\varphi^\eps$ in a random travel time frame, as explained in section \ref{sect:Ito-Schr}. 

\subsection{Random central axis travel time analysis}
\label{sect:TravelT}
{We obtain from equations \eqref{eq:In1} and \eqref{eq:F2} that the random velocity along the axis of the beam is given by
\begin{equation}
\frac{c_o}{c^\eps({\bf 0},z)} = \sqrt{1 + \mu^\eps(\vx)} \sim   1 + \frac{\eps^3}{2} \mu\Big({\bf 0},\frac{z}{\eps^2} \Big), \qquad \mbox{as}~ \eps \to 0,
\label{eq:RT1}
\end{equation}
so the central  axis travel time is 
\begin{equation}
\int_0^z \frac{dz'}{c^\eps({\bf 0},z')} \sim  \frac{z}{c_o} +  \frac{\eps^4 {\cal Z}^\eps(z)}{c_o}, \qquad 
\rc{ {\cal Z}^\eps(z) = \frac{1}{2 \eps} \int_0^{z}  dz' \, \mu\big({\bf 0}, \frac{z'}{\eps^2}\big)},
\label{eq:RT3} 
\end{equation}
and has random fluctuations modeled by ${\cal Z}^\eps$.}
Due to the high frequency $\om=\frac{\varOmega}{\eps^4}$, these fluctuations have a significant effect on the phase of the wave field
\begin{equation}
\frac{\varOmega}{\eps^4} \int_0^z \frac{dz'}{c^\eps({\bf 0},z')} \sim \frac{k(\varOmega) z}{\eps^4} + k(\varOmega) {\cal Z}^\eps(z),
\label{eq:RT4}
\end{equation}
and the next proposition describes the asymptotics of ${\cal Z}^\eps$, as $\eps \to 0$.

\vspace{0.05in}

\begin{prop}
\label{prop.1}
The random process ${\cal Z}^\eps$ defined in \eqref{eq:RT3}
satisfies
\begin{equation}
{\cal Z}^\eps \Big(\eps^{2 \alpha/(1+\alpha)}z \Big) \to C_H W^H(z), \quad \mbox{as}~ \eps \to 0, 
\label{eq:RT5}
\end{equation}
where the convergence is in distribution, $W^H(z)$ is a fractional Brownian motion with Hurst index 
$H = (1+\alpha)/2$,
and
$
C_H  = \frac{1}{2 \pi} \sqrt{\frac{C_\alpha}{\alpha(\alpha+1)}}, 
$
with $C_\alpha$ given  in \eqref{eq:ST5}.
At $O(1)$ range the process ${\cal Z}^\eps$ satisfies 
\begin{equation}
\eps^\alpha {\cal Z}^\eps(z) \to C_H W^H(z), \quad \mbox{as}~ \eps \to 0,
\label{eq:RT8}
\end{equation}
where the convergence is in distribution and the limit is as in \eqref{eq:RT5}.
\end{prop}

\vspace{0.05in} \begin{proof} The convergence is proved in \cite{marty04} for a Gaussian $\mu$. The result extends to a process $\mu$ given by a smooth and bounded function of a Gaussian process 
as shown in \cite{marty11travelt} where the precise conditions on the function are given.  \end{proof}

\vspace{0.04in}
 We recall from \cite{mandelbrot1968fractional} that the fractional Brownian motion $W^H$ is a Gaussian process, with stationary increments, satisfying 
\begin{equation}
\EE \big[W^H(z) \big] = 0, \quad~ \EE \big[ W^H(z) W^H(z')\big] = \frac{1}{2} \big[z^{2H} + (z')^{2H} + |z-z'|^{2H} \big].
\label{eq:RT9}
\end{equation}
The proposition says that: 

\vspace{0.04in}
\begin{enumerate}
\itemsep 0.04in
\item The process ${\cal Z}^\eps(z)$ and therefore the phase \eqref{eq:RT4} are randomized i.e., have significant random fluctuations, on a short $O(\eps^{2 \alpha/(1+\alpha)})$ range scale. 
\rc{In the physical variables (\ref{eq:scaling}),
this corresponds to a propagation distance which is such that 
$k(\omega_o^\eps)^2 \EE[ (\eps^{4} {\cal Z}^\eps(z))^2]\sim 1$, that is to say,
$z \sim  [ k(\omega_o^\eps)^2 \chi_\alpha^\eps ] ^{-1/(1+\alpha)}$.}
\item  Even though $\mu$ is not a Gaussian process, the phase fluctuations are Gaussian.
\item  The random fluctuations of the phase are huge i.e., $O(\eps^{-\alpha})$ at $O(1)$ range and must be removed in order to characterize the 
$\eps \to 0$ limit of $\varphi^\eps$.
\end{enumerate}

\subsection{Wave in the random travel time frame}
\label{sect:Ito-Schr}
After removing the random phase, which is equivalent to observing the wave in the central axis random time frame ${\cal Z}^\eps/c_o$, we get that 
\begin{equation}
\psi^\eps(\varOmega,\bX,z) = \varphi^\eps(\varOmega,\bX,z) \exp \big[-i k(\varOmega) {\cal Z}^\eps(z) \big],
\label{eq:IS1}
\end{equation}
satisfies the paraxial equation 
\begin{align}
\left[ 2 i k(\varOmega) \partial_z + \Delta_\bX + \frac{k^2(\varOmega)}{\eps} \nu \Big(\bX,\frac{z}{\eps^2} \Big) \right]  \psi^\eps (\varOmega,\bX,z) &= 0,  \quad z > 0, \label{eq:ParaxPsi} \\
 \psi^\eps (\varOmega,\bX,z=0) &= \hat F(\varOmega,\bX), \label{eq:ParaxPsiIC}
\end{align}
with the random potential 
\begin{equation}
\nu (\bX,z ) = \mu  (\bX,z) - \mu ({\bf 0},z).
\label{eq:defNu}
\end{equation}

The process $\nu$ is stationary in $z$, but not in $\bX$, and we explain next that its covariance is integrable in $z$. 
Indeed, 
\begin{align}
\mbox{Cov}_\nu(\bX,\bX',z-z') = \EE \big[ \nu(\bX,z) \nu(\bX',z')\big] = \mbox{Cov}_\mu(\bX-\bX',z-z')  \nonumber \\
+ \mbox{Cov}_\mu({\bf 0},z-z')- \mbox{Cov}_\mu(\bX',z-z') -  \mbox{Cov}_\mu(\bX,z-z') ,
\end{align}
 and using  equation \eqref{eq:ST1} we get for $\bX = \bX'$:
\begin{align}
\mbox{Cov}_\nu(\bX,\bX,z) =& \frac{\chi_\alpha |z|^{\alpha-1}}{\pi^2}\left[ \int_0^{|z|/l_o} d u \, u^{-\alpha} 
\mbox{sinc} (u) \right.\nonumber \\&\left.- \Big(1 + \frac{|\bX|^2}{z^2}\Big)^{(\alpha-1)/2} \int_0^{|z|/l_o\sqrt{1+|\bX|^2/z^2}} d u \, u^{-\alpha} 
\mbox{sinc} (u)\right]. \label{eq:Covnu}
\end{align}
We are interested in the decay of this expression at $|z| \to \infty$, which can be seen from the asymptotic expansion 
\begin{align}
\mbox{Cov}_\nu(\bX,\bX,z) &\sim \frac{C_\alpha}{\pi^2} |z|^{\alpha-1} \Big[1 - \Big(1 + \frac{|\bX|^2}{z^2}\Big)^{(\alpha-1)/2} \Big] \nonumber \\
&\sim  \frac{C_\alpha (1-\alpha)|\bX|^2}{2 \pi^2} |z|^{\alpha-3}, \quad \mbox{as} ~ |z| \to \infty,
\end{align}
with  constant $C_\alpha$ given by \eqref{eq:ST5}. {Since $\alpha \in (0,1)$, the} decay in $|z|$ is fast enough to make the covariance integrable, 
and we say that  the process $\nu$ is mixing.


\vspace{0.04in}
\begin{prop}
\label{prop.2} 
The solution 
$\psi^\eps$ of (\ref{eq:ParaxPsi}-\ref{eq:ParaxPsiIC}) converges in distribution, in the space $C([0,+\infty),L^2(\RR \times \RR^2,\CC))$ of continuous functions of $z \in [0,\infty)$ that are square integrable in $(\varOmega,\bX)$,   to the solution of the 
{It\^o-Schr\"odinger equation} 
\begin{equation}
    \label{eq:itoschro}
 { d {\psi}(\varOmega,\bX,z) =   \frac{i}{2k(\varOmega)}  \Delta_\bX {\psi}(\varOmega,\bX,z) dz  + \frac{i k(\varOmega)}{2}   {\psi}(\varOmega,\bX,z) \circ d{W}(\bX,z)  }   , 
 \end{equation}
with initial condition 
\begin{equation}
\psi(\varOmega,\bX,z=0) = \hat F(\varOmega,\bX). \label{eq:itoschroIC}
\end{equation}
 The symbol ``$\circ$" denotes the Stratonovich integral and 
$W(\bX,z)$ is  a  centered Brownian field.
It satisfies $\EE[ W(\bX,z) W(\bX',z') ] = \gamma(\bX,\bX') \, \min( z , z')$, with 
\begin{equation}
\gamma(\bX,\bX') = 
\frac{\chi_\alpha}{2\pi} 
\int_0^{l_0^{-1}} d \kappa \, 
\big[ J_0(\kappa |\bX-\bX'|) +1 - J_0(\kappa |\bX|)-J_0( \kappa |\bX'|) \big] \label{eq:defgamma}
\kappa^{-1-\alpha},
\end{equation}
where $J_0$ is the Bessel function of the first kind and of order $0$.
\end{prop}

\begin{proof}
This theorem was proved for a fixed frequency in \cite{fannjiang2004scaling}:
For any $\varOmega \neq 0$, 
the solution $(z,\bX)\mapsto \psi^\eps(\varOmega,\bX,z)$ of (\ref{eq:ParaxPsi})
converges in distribution, in the space $D([0,+\infty),L^2(\RR^2,\CC))$, to the solution  $(z,\bX)\mapsto \psi(\varOmega,\bX,z)$  of 
(\ref{eq:itoschro}). Here $D$ is the space of c\`adl\`ag functions.
The proof can be  extended to the multi-frequency case as follows: for any set of non-zero frequencies $(\varOmega_j)_{j=1}^n$, the random process $(z,\bX)\mapsto (\psi^\eps(\varOmega_j,\bX,z))_{j=1}^n$
converges in distribution in $D([0,+\infty),L^2(\RR^2,\CC^n))$ to the process $(z,\bX)\mapsto (\psi(\varOmega_j,\bX,z))_{j=1}^n$.

The tightness of $\psi^\eps$ in $D([0,+\infty),L^2_w(\RR\times \RR^2,\CC))$ (with $L^2_w$ equipped with the weak topology) can be established as in \cite[Section 3.1]{fannjiang2004scaling} by using the tightness criterion  \cite[Chap. 3, Theorem 4]{kushner}.
This proves the convergence of $\psi^\eps$ to $\psi$ in the space $D([0,+\infty),L^2_w(\RR\times \RR^2,\CC))$.
Since both the original and limit processes preserve the $L^2$-norm of the initial data, the process converges in 
$D([0,+\infty),L^2(\RR\times \RR^2,\CC))$.
Furthermore, since both the original and limit processes are continuous,
the convergence actually holds in $C([0,+\infty),L^2(\RR\times \RR^2,\CC))$.
\end{proof}

\section{Application of the asymptotic analysis}
\label{sect:Application}
We now use the asymptotic results stated in Propositions \ref{prop.1} and \ref{prop.2} to analyze  the coherent wave (subsection 
\ref{sect:CohWave}) and the spatial and frequency covariance of $\varphi^\eps$ (subsections \ref{sect:CovSpatial}--\ref{sect:CovFreq}) 
in the limit $\eps \to 0$.
We also characterize in subsection \ref{sect:Pulse} the deformation of the pulse emitted by the source, induced by scattering in the random medium.

\subsection{The coherent wave}
\label{sect:CohWave}
Scattering causes a loss of coherence of the wave field, which manifests  as an exponential  decay  
of the mean wave (aka coherent wave) $\EE[\varphi^\eps]$ with respect to the range $z$.  The length scale of decay, called the scattering mean free path, 
gives the range limit at which conventional methods\footnote{Conventional methods are based on the assumption that the medium through which the waves propagate is homogeneous or more generally, known and non-scattering. }  used for imaging and free space communication are useful in random media.

The leading factor in the loss of coherence of $\varphi^\eps$ is the random phase $k {\cal Z}^\eps$, which becomes significant at  $O(\eps^{2 \alpha/(1+\alpha)})$ range. Indeed, Propositions \ref{prop.1} and \ref{prop.2} give that 
\begin{equation}
\EE\big[ \exp(i k(\varOmega) {\cal Z}^\eps(\eps^{2\alpha/(1+\alpha)} z) \big] \stackrel{\eps \to 0}{\longrightarrow} \exp\Big[ - \frac{C_H^2 k^2(\varOmega) z^{2H}}{2} \Big]
\label{eq:LC1}
\end{equation}
and  
\begin{equation}
\EE\big[ \varphi^\eps(\varOmega,\bX,\eps^{2\alpha/(1+\alpha)} z) \big] \stackrel{\eps \to 0}{\longrightarrow} \hat F(\varOmega,\bX) \exp\Big[ - \frac{C_H^2 k^2(\varOmega) z^{2H}}{2} \Big], \label{eq:LC2}
\end{equation}
so the scattering mean free path  has the asymptotic expansion 
\begin{equation}
\mathscr{S}_{\varphi^\eps}(\varOmega) \sim  \eps^{2\alpha/(1+\alpha)} [ C_H k(\varOmega) ]^{-1/H}.
\label{eq:LC3}
\end{equation}

However, the wave $\psi^\eps$ defined in \eqref{eq:IS1} by removing the large random phase $k {\cal Z}^\eps$ from $\varphi^\eps$, maintains its coherence up to a much longer, $O(1)$ range.  Proposition \ref{prop.2} gives that 
\begin{equation}
\EE[\psi^\eps(\varOmega,\bX,z)]   \stackrel{\eps \to 0}{\longrightarrow}  M_1(\varOmega,\bX,z),
\label{eq:defM1}
\end{equation}
where $M_1$ solves the evolution equation
 \begin{equation}
    \partial_z M_1(\varOmega,\bX,z) =   \frac{i}{2k(\varOmega)}  \Delta_\bX M_1(\varOmega,\bX,z) - \frac{k^2(\varOmega)}{4} \Theta(\bX)  M_1(\varOmega,\bX,z)  ,     \label{eq:evolcoh}
\end{equation}
obtained by taking the expectation in \eqref{eq:itoschro}, with initial condition derived from \eqref{eq:itoschroIC}
 \begin{equation}
 \quad M_1 (\varOmega,\bX,z=0) =  \hat F(\varOmega,\bX) ,
    \label{eq:evolcohIC}
\end{equation}
and with damping coefficient
\begin{align}
\label{def:Theta}
\Theta(\bX) &=  \frac{\gamma(\bX,\bX)}{2} = \frac{\chi_\alpha}{2\pi} 
\int_0^{l_0^{-1}} d \kappa \, \big[ 1 - J_0(\kappa |\bX|)\big] \kappa^{-1-\alpha}  \nonumber \\
&= \frac{\chi_\alpha|\bX|^\alpha}{2\pi} 
\int_0^{|\bX|/l_o} d s \, \big[ 1 - J_0(s)\big] s^{-1-\alpha}.
\end{align}
The damping models the loss of coherence of $\psi^\eps$. It is weaker at the axis of the beam and it increases away from it.  In fact, at  $|\bX|/l_o \to \infty$  we get the asymptotic expansion 
\begin{align}
\label{def:Gamma}
\Theta(\bX)  \sim d_\alpha |\bX|^\alpha, \quad d_\alpha = \frac{\chi_\alpha}{2\pi} 
\int_0^{\infty} d s \, \big[ 1 - J_0(s)\big] s^{-1-\alpha} = \frac{\chi_\alpha}{2^{1+\alpha} \pi} \frac{\Gamma(1-\alpha/2)}{\alpha \Gamma(1+\alpha/2)}.
\end{align}

\subsection{Spatial covariance}
\label{sect:CovSpatial}
Although the wave loses its coherence (the mean wave decays with the propagation distance), 
wave energy is not lost but converted into incoherent, zero-mean fluctuations.
These incoherent waves can be characterized by the second-order moments of the wave field, that we analyze in this subsection and the next ones.
For imaging purposes, 
it is possible to extract information from the observation 
of the incoherent waves and their correlation properties
in space and frequency. An example of exploiting such knowledge is the  coherent interferometric (CINT) methodology 
for robust imaging in random media \cite{borcea2021imaging,borcea2011enhanced,borcea2006adaptive}. 

There are two intrinsic scales that capture the decorrelation properties of the wave field: the ``decoherence length",
which is the length scale of decay of the covariance of $\varphi^\eps$ over cross-range offsets and the ``decoherence frequency", 
which is the frequency scale of decay of the covariance over frequency offsets. In this subsection  we study the spatial 
covariance i.e.,  fix the frequency at $\varOmega$,  and estimate the decoherence length. 
We note from definition \eqref{eq:IS1} that the phase $k {\cal Z}^\eps$ plays no role in the spatial covariance, 
\begin{align*}
\EE\big[\varphi^\eps(\varOmega,\bX_1, z) \overline{\varphi^\eps(\varOmega,\bX_2,z)} \big] = \EE\big[\psi^\eps(\varOmega,\bX_1, z) \overline{\psi^\eps(\varOmega,\bX_2,z)} \big] 
\stackrel{\eps \to 0}{\longrightarrow}
 {\cC}_\varOmega(\bX_1,\bX_2, z).
\end{align*}
Here the bar stands for the complex conjugate, the notation $\cC_\varOmega$ emphasizes that the frequency is fixed at $\varOmega$,  and the $\eps \to 0$ limit 
\begin{equation}
 {\cC}_\varOmega(\bX_1,\bX_2, z) 
=\EE \big[ \psi(\varOmega,\bX_1,z) \overline{\psi(\varOmega,\bX_2,z)}\big]
\end{equation}
is obtained from the It\^{o}-Schr\"{o}dinger equation  in Proposition \ref{prop.2}. Using the identity 
\[
\gamma(\bX_1,\bX_2) - \Theta(\bX_1) - \Theta(\bX_2) = - \Theta(\bX_1-\bX_2),
\]
deduced from definitions \eqref{eq:defgamma} and \eqref{def:Theta}, we get the  evolution  equation 
\begin{align}
\partial_z \cC_\varOmega(\bX_1,\bX_2,z) = \left[\frac{i}{2 k(\varOmega)} \big(\Delta_{\bX_1} - \Delta_{\bX_2}\big) - \frac{k^2(\varOmega)}{4} \Theta(\bX_1-\bX_2) \right]{\cC}_\varOmega(\bX_1,\bX_2, z), \label{eq:SPC1}
\end{align}
for $z > 0$, with initial condition
\begin{equation}
\cC_\varOmega(\bX_1,\bX_2,z=0) = \hat F(\varOmega,\bX_1) \overline{\hat F(\varOmega,\bX_2)}.
\label{eq:SPC2}
\end{equation}

We can solve equation \eqref{eq:SPC1} explicitly, by changing coordinates 
\begin{equation}
\label{eq:coordC}
(\bX_1,\bX_2) \mapsto (\obX,\tbX), \quad \obX = \frac{1}{2}(\bX_1 + \bX_2), \quad \tbX = \bX_1-\bX_2,
\end{equation}
and then taking the Fourier transform with respect to the offset vector $\tbX$, which defines the 
mean Wigner transform 
\begin{equation}
\cW_\varOmega(\obX,\bkappa,z) = \int_{\RR^2} d \tbX \, \cC_{\varOmega} \Big(\obX + \frac{\tbX}{2},\obX-\frac{\tbX}{2},z\Big) e^{-i \bkappa \cdot \tbX}.
\label{eq:SPC3}
\end{equation}
This transform is important by itself, as it tells us how the energy at $\obX$ is distributed over the directions   i.e., along $\bkappa$.
It plays a key role in the analysis of imaging and time reversal methods in random media 
\cite{borcea2011enhanced,borcea2007asymptotics,papanicolaou2007self}.
The calculation of $\cW_\varOmega$ is given in Appendix \ref{ap:A} and the result is stated in the following proposition: 

\vspace{0.04in}
\begin{prop}
\label{prop.3}
The mean Wigner transform is given by 
\begin{align}
\nonumber
\cW_\varOmega(\obX,\bkappa,z) =&\frac{1}{(2\pi)^2}
\int_{\RR^2} d \bq \int_{ \RR^2}  d \tbX \, \exp \Big[ i\bq \cdot\Big(\obX-\bkappa \frac{z}{k(\varOmega)}\Big)-i \bkappa\cdot\tbX \Big] \hat{\cW}_{\varOmega,0}(\bq,\tbX) \\
&\qquad \qquad \qquad \times
\exp\Big[ -\frac{k^2(\varOmega)}{4}\int_0^z dz' \, \Theta \Big(\tbX+\frac{\bq z'}{k(\varOmega)} \Big)\Big] ,
\label{eq:expressW2}
\end{align}
with 
\begin{equation}
\hat{\cW}_{\varOmega,0}(\bq,\tbX) = \int_{\RR^2} d \obX \, \hat F\Big(\varOmega, \obX + \frac{\tbX}{2}\Big) \overline{\hat F\Big(\varOmega,\obX - \frac{\tbX}{2}\Big)} 
e^{-i \bq \cdot \obX}. \label{eq:expressW2IC}
\end{equation}

\end{prop}

\vspace{0.04in} The spatial covariance is obtained from the expression \eqref{eq:expressW2} using the inverse Fourier transform
\begin{align}
\nonumber
\cC_{\varOmega} \Big(\obX + \frac{\tbX}{2},\obX-\frac{\tbX}{2},z\Big) &= \frac{1}{(2 \pi)^2} \int_{\RR^2} d \bkappa \, \cW_\varOmega(\obX,\bkappa,z) e^{i \bkappa \cdot \tbX} \\&= \frac{1}{(2 \pi)^2} \int_{\RR^2} d \bq \, \hat{\cW}_{\varOmega,0} \Big(\bq ,\tbX - \frac{\bq z}{k(\varOmega)} \Big) e^{i \bq \cdot \bX} \nonumber \\
&\quad  \times \exp \Big[- \frac{k^2(\varOmega)}{4}\int_0^z dz' \, \Theta \Big(\tbX-\frac{\bq (z-z')}{k(\varOmega)}\Big)\Big], \label{eq:Covariance}
\end{align} 
and we study it next using the asymptotic expansion \eqref{def:Gamma} of $\Theta$, which {holds when its argument is much larger than $l_o$.}
Note that the coefficient $d_\alpha$ in this expansion quantifies the strength of the fluctuations in the random medium. 

{We have already assumed a large outer scale $L_o$. We now consider, in addition, 
 a strong fluctuation and small inner scale regime, in the sense 
\begin{equation}
l_o \ll \frac{1}{Q(z)}   \ll r_{\rm s} \ll R(z),
\label{eq:SMF} 
\end{equation}}
{where we recall that $r_{\rm s}$ is the initial radius of the beam.  There are two new  scales in equation \eqref{eq:SMF}: the range dependent beam radius}
\begin{equation}
R(z) = \left[\frac{d_{\alpha} k^{2-\alpha}(\varOmega) z^{\alpha+1}}{\alpha+1}\right]^{1/\alpha},
\label{eq:defR}
\end{equation}
{which quantifies the  spatial support of the mean intensity (subsection \ref{sect:MInt}), 
and the range dependent wave vector radius}
\begin{equation}
Q(z) = \left[d_{\alpha} k^{2}(\varOmega) z\right]^{1/\alpha},
\label{eq:defQ}
\end{equation}
{which quantifies the  wave vector support of the mean spectrum (subsection \ref{sect:MSpect}).}

\subsubsection{The mean intensity} \label{sect:MInt}
The mean intensity $\EE\left[ |\psi(\varOmega,\obX,z)|^2\right]$ is equal to ${\cal C}_\varOmega(\obX,\obX,z)$. From Proposition \ref{prop.3} we obtain the following result.
\begin{proposition}
In the regime (\ref{eq:SMF}), 
the mean intensity has the form
\begin{align}
\EE\left[ |\psi(\varOmega,\obX,z)|^2\right] & \simeq \frac{\hat \cW_{\varOmega,0}({\bf 0},{\bf 0})}{R^2(z)} \Psi_\alpha \Big(\frac{\obX}{R(z)}\Big),
\label{eq:MeanInt}
\end{align}
with   $\hat{\cW}_{\varOmega,0} $ given by \eqref{eq:expressW2IC} and
\begin{align}
\Psi_\alpha (\bxi) &= \frac{1}{(2\pi)^2} \int_{\RR^2} d \bu \,  e^{i \bu \cdot \bxi - \frac{ |\bu |^{\alpha}}{4}} = 
\frac{1}{2 \pi} \int_0^\infty d \eta  \, \eta J_0(|\bxi| \eta ) e^{- \frac{ \eta^{\alpha}}{4}}. \label{eq:defPsialpha}
\end{align}
\end{proposition}

\begin{proof}
Setting $\tbX = {\bf 0}$ in  \eqref{eq:Covariance} and using the asymptotic expansion \eqref{def:Gamma}  we obtain 
the following expression of the mean intensity
\begin{align*}
\EE\left[ |\psi(\varOmega,\obX,z)|^2\right] & = \frac{1}{(2 \pi)^2} \int_{\RR^2} d \bq \, \hat{\cW}_{\varOmega,0} \Big(\bq ,- \frac{\bq z}{k(\varOmega)} \Big) 
\exp \Big[i \bq \cdot \bX - \frac{R^\alpha(z) |\bq |^{\alpha}}{4} \Big] \nonumber \\
&= \frac{1}{(2 \pi)^2 R^2(z)} \int_{\RR^2} d \bu  \, \hat{\cW}_{\varOmega,0} \Big(\frac{\bu}{R(z)} ,- \frac{\bu z}{k(\varOmega)R(z)} \Big) \exp \Big[i  \frac{\bu \cdot \bX}{R(z)} - \frac{ |\bu |^{\alpha}}{4} \Big], 
\end{align*}
where we let  $\bq = \bu/R$,  with $R$ defined in \eqref{eq:defR}. Due to the exponential, only $|\bu| = O(1)$ contributes to the integral, so the arguments of $\hat{\cW}_{\varOmega,0}$ satisfy 
\begin{equation}
\frac{|\bu|}{R(z)} = O \left(R^{-1}(z)\right) \ll r_{\rm s}^{-1}, \qquad \frac{|\bu| z}{k(\varOmega)R(z)} = O \left(\frac{z}{k(\varOmega)R(z)}\right) \ll r_{\rm s}.
\label{eq:Ineq1}
\end{equation}
Here we used the assumption \eqref{eq:SMF} and the second inequality is because by definitions (\ref{eq:defR}-\ref{eq:defQ}) we have 
\[
\frac{z}{k(\varOmega)R(z)} = \frac{(\alpha+1)^{1/\alpha}}{[d_{\alpha} k^{2}(\varOmega) z]^{1/\alpha}} = \frac{(\alpha+1)^{1/\alpha}}{Q(z)} \ll r_{\rm s}.
\]
 We infer from definition \eqref{eq:expressW2IC} of $\hat{\cW}_{\varOmega,0} $ that its support in the first argument is at wave vectors with 
$O(r_{\rm s}^{-1})$ norm and the support in the second argument is at cross-range vectors of $O(r_{\rm s})$ norm.  Thus, due to the inequalities 
\eqref{eq:Ineq1}, we can approximate the mean intensity by 
(\ref{eq:MeanInt}).
\end{proof}

We plot the function $\Psi_\alpha$ in section \ref{sect:Comparison}. It peaks at the origin and it is negligible outside a disk of $O(1)$ radius. It is smooth at ${\bf 0}$ and can be expanded as
\begin{equation}
\label{eq:expandpsialpha}
\Psi_\alpha (\bxi)
=
\Psi_\alpha ({\bf 0}) \big[ 1 -q_\alpha |\bxi|^2 +o(|\bxi|^2) \big]  ,
\end{equation}
where 
\begin{align}
\Psi_\alpha({\bf 0})=\frac{2^{{4}/{\alpha}} \Gamma(2/\alpha)}{2\pi \alpha} ,\quad \quad 
q_\alpha = 2^{4/\alpha-2} \frac{\Gamma(4/\alpha) }{\Gamma(2/\alpha)}. 
\end{align}
Therefore, the scale $R(z)$ quantifies the support of the mean intensity, and  we call it the ``beam radius" at range $z$. If there where no random medium, beam broadening would be entirely due to diffraction. Here the broadening is caused by  scattering in the random medium and it is significant, because $R(z)$ is  much larger than the initial radius $r_{\rm s}$ of the beam, per equation \eqref{eq:SMF} and 
with a growth rate in  $z$  that is higher than in the homogeneous medium.  

\subsubsection{The mean spectrum} \label{sect:MSpect} 
Using the Fourier transform 
\[
\hat \psi(\varOmega,\bkappa,z) = \int_{\RR^2} d \bX \, \psi(\varOmega,\bX,z) e^{-i \bkappa \cdot \bX},
\]
the change of coordinates \eqref{eq:coordC} and the definition \eqref{eq:SPC3} of the Wigner transform, we can calculate the mean spectrum as follows
\begin{align*}
\EE \big[|\hat \psi(\varOmega,\bkappa,z)|^2 \big] &= \int_{\RR^2} d \bX_1 \int_{\RR^2} d \bX_2 \, \EE \big[\psi(\varOmega,\bX_1,z) \overline{\psi(\varOmega,\bX_2,z)}\big] e^{i \bkappa 
\cdot (\bX_2-\bX_1)} \nonumber \\
&=  \int_{\RR^2} d \obX \int_{\RR^2} d \tbX \, \cC_\varOmega \Big(\obX + \frac{\tbX}{2},\obX - \frac{\tbX}{2},z \Big) e^{-i \bkappa \cdot \tbX} \nonumber \\
&= \int_{\RR^2} d \obX \, \cW_\varOmega(\obX,\bkappa,z),
 \end{align*}
with right-hand side given in Proposition \ref{prop.3}. 
We then obtain the following result.
\begin{proposition}
In the regime (\ref{eq:SMF}), the mean spectrum is of the form
\begin{align}
\EE \big[|\hat \psi(\varOmega,\bkappa,z)|^2\big] & \simeq \frac{(2 \pi)^2 \hat \cW_{\varOmega,0}({\bf 0},{\bf 0})}{Q^2(z)} \Psi_\alpha \Big(\frac{\bkappa}{Q(z)} \Big),
\label{eq:meanSpectr}
\end{align}
with $\Psi_\alpha$ defined in \eqref{eq:defPsialpha}.
\end{proposition}
\begin {proof}
Using the asymptotic expansion  \eqref{def:Gamma} of $\Theta$ and integrating over $\obX$ and $\bq$ we get 
\begin{align*}
\EE \big[|\hat \psi(\varOmega,\bkappa,z)|^2 \big] & =  \int_{\RR^2} d \tbX \, \hat \cW_{\varOmega,0}({\bf 0},\tbX )\exp\Big[ - i \bkappa \cdot \tbX - 
\frac{d_\alpha k^2(\varOmega)  z |\tbX|^\alpha}{4} \Big] \nonumber \\
&=\frac{1}{Q^2(z)} \int_{\RR^2} d \bu \, \hat \cW_{\varOmega,0} \Big({\bf 0},\frac{\bu}{Q(z)} \Big) \exp\Big[- i \frac{\bkappa \cdot \bu}{Q(z)} - \frac{|\bu|^\alpha}{4} \Big],
\end{align*}
with $Q$ defined in \eqref{eq:defQ}. Arguing as before, since only $|\bu| = O(1)$ contributes to the integral, due to the exponential, the argument of 
$\hat \cW_{\varOmega,0}$ satisfies
\[
\frac{|\bu|}{Q(z)} = O\left(Q^{-1}(z)\right) \ll r_{\rm s},
\]
and we can approximate the mean spectrum by 
 (\ref{eq:meanSpectr}).
\end{proof}

This result shows that the scale $Q$ quantifies the support of the spectrum, so we call it the ``spectral radius" at range $z$. The initial spectral radius is $O(r_{\rm s}^{-1})$, 
but due to scattering  in the random medium it becomes significantly larger at $O(1)$ range, per equation \eqref{eq:SMF}.  This goes hand in hand with 
 the broadening of the beam described by equation \eqref{eq:MeanInt}.

\subsubsection{The spatial covariance function}
In the strong fluctuation regime (\ref{eq:SMF}) it is possible to express the covariance in terms of the beam radius $R$ and wave vector radius $Q$ as follows.

\begin{proposition}
\label{prop:SpCov}
In the regime (\ref{eq:SMF}), 
the covariance has the form
\begin{align}
\cC_\varOmega\Big(\obX+\frac{\tbX}{2},\obX-\frac{\tbX}{2},z\Big) \simeq \frac{\hat \cW_{\varOmega,0}({\bf 0},{\bf 0})}{R^2(z)} \Phi_\alpha \Big(\frac{\obX}{R(z)},
\tbX Q(z) \Big), \label{eq:CovSMF}
\end{align}
with the function
\begin{align}
\Phi_\alpha(\bxi,\bzeta) = \frac{1}{(2\pi)^2} \int_{\RR^2} d \bu \, \exp \left[ i \bu \cdot \bxi - \frac{(1+\alpha)}{4} 
\int_0^1 ds \, \Big| \frac{\bzeta}{(1+\alpha)^{1/\alpha}} - \bu s \Big|^{\alpha} \right].
\label{eq:defPhiAlpha}
\end{align}
\end{proposition}

\begin{proof}
Starting from equation \eqref{eq:Covariance}, using the asymptotic expansion \eqref{def:Gamma}, changing variables as 
$
s =(z-z')/{z}$ and $ \bu = R \bq,
$
and using definitions \eqref{eq:defR} and \eqref{eq:defQ} we get 
\begin{align*}
\cC_\varOmega\Big(\obX+\frac{\tbX}{2},\obX-\frac{\tbX}{2},z\Big) = \frac{1}{(2 \pi)^2R^2(z)} \int_{\RR^2}d \bu \, 
\hat \cW_{\varOmega,0}\Big( \frac{\bu}{R(z)},\tbX - \frac{\bu z}{k(\varOmega) R(z)} \Big) \nonumber \\
\times 
\exp \left[ i \frac{\bu \cdot \obX}{R(z)} - \frac{(1+\alpha)}{4} \int_{0}^1 ds \, \left| \frac{\tbX Q(z)}{(1+\alpha)^{1/\alpha}} - \bu s \right|^\alpha\right].
\end{align*}
Again, we conclude that only $|\bu| = O(1)$ contributes to the integral, due to the decaying exponential, so under the strong fluctuations assumption
\eqref{eq:SMF} we can make the approximation
\[
\hat \cW_{\varOmega,0}\Big( \frac{\bu}{R(z)},\cdot\Big) \approx \hat \cW_{\varOmega,0}({\bf 0},\cdot).
\]
We also get from definitions (\ref{eq:defR}-\ref{eq:defQ}) and the assumption \eqref{eq:SMF} the following estimates
\[
\frac{|\bu| z}{k(\varOmega) R(z)} = O \left(Q^{-1}(z)\right) \ll r_{\rm s}, \qquad |\tbX| = O \left(Q^{-1}(z)\right) \ll r_{\rm s}.
\]
Here we used that  $|\tbX| Q = O(1)$ in order for the exponential to be large. Therefore, the covariance can be approximated by (\ref{eq:CovSMF}).
\end{proof}

Note that  $\Phi_\alpha(\bxi,{\bf 0}) = \Psi_\alpha(\bxi)$, 
with $\Psi_\alpha$ given in \eqref{eq:defPsialpha}, and we also have 
\begin{equation}
\int_{\RR^2} d \bxi \int_{\RR^2} d \bzeta \, \Phi_\alpha(\bxi,\bzeta) e^{-i \bkappa \cdot \bzeta} = (2 \pi)^2 \Psi_\alpha (\bkappa).
\end{equation}
Contrarily to the function $\bxi \mapsto \Phi_\alpha(\bxi,{\bf 0})$ that is smooth at ${\bf 0}$ by (\ref{eq:expandpsialpha}), the function $\bzeta \mapsto \Phi_\alpha({\bf 0},\bzeta)$ has a cusp at ${\bf 0}$ (see Appendix~\ref{app:C}):
\begin{equation}
\label{eq:expandPhizeta}
\Phi_\alpha({\bf 0},\bzeta) = \Phi_\alpha({\bf 0},{\bf 0}) 
\Big( 1- r_\alpha |\bzeta|^{\alpha+1} +o(|\bzeta|^{\alpha+1})\Big),
\end{equation}
where $r_\alpha$ is given by 
\begin{equation}
r_\alpha = \frac{\alpha}{2^{2+2/\alpha} (1+\alpha)^{1/\alpha+2}}
\frac{\Gamma(1/\alpha) \Gamma(1/2-\alpha/2)\Gamma(1+\alpha/2)}{\Gamma(2/\alpha) \Gamma(1/2+\alpha/2) \Gamma(1-\alpha/2)}.
\end{equation}
This implies that the covariance (\ref{eq:CovSMF}) has a cusp at ${\bf Y}= {\bf 0}$.

We plot the marginals $\bxi \mapsto \Phi_\alpha(\bxi,{\bf 0})$ and $\bzeta \mapsto \Phi_\alpha({\bf 0},\bzeta)$ in section 
\ref{sect:Comparison}. 
They peak at the origin and are negligible outside a disk with $O(1)$ radius.
We conclude therefore, from \eqref{eq:CovSMF},  that $Q^{-1}$ quantifies the length scale of decorrelation over the spatial offsets $\tbX$ at range $z$, so we can refer to it as the ``decoherence length"
\begin{equation}
\mathfrak{X}(z) = Q^{-1}(z) = O \left( (d_\alpha z)^{-1/\alpha} k^{-2/\alpha}(\varOmega) \right).
\end{equation}
This is proportional to the wavelength raised to the power $2/\alpha$, and it decreases with the range $z$ and with the 
strength of the random medium, quantified by $d_\alpha$.

\subsection{The frequency covariance function}
\label{sect:CovFreq}
The leading factor in the frequency decorrelation of $\varphi^\eps$ at $O(1)$ range is the random phase $k (\varOmega){\cal Z}^\eps$. Indeed, we obtain from Propositions \ref{prop.1}--\ref{prop.2} that this phase gives a significant contribution to the covariance for $O(\eps^\alpha)$ frequency offsets,
\begin{align}
\EE \bigg[ \varphi^\eps \Big(\varOmega + \frac{\eps^\alpha \tilde{\varOmega}}{2}, \obX + \frac{\tbX}{2},z\Big) \overline{\varphi^\eps \Big(\varOmega - \frac{\eps^\alpha \tilde{\varOmega}}{2}, \obX -\frac{\tbX}{2},z\Big)}\bigg]\stackrel{\eps \to 0}{\longrightarrow} \exp\bigg[ - \frac{\tilde{\varOmega}^2 C_H^2 z^{2 H}}{2c_o^2}\bigg] \nonumber \\
 \times \cC_\varOmega\Big(\obX + \frac{\tbX}{2},\obX - \frac{\tbX}{2},z\Big). \label{eq:CF1}
\end{align}
This contribution is described by the Gaussian in $\tilde{\varOmega}$, whose standard deviation  defines the decoherence frequency
\begin{equation}
\varOmega_{\varphi^\eps}(z)= \frac{c_o \eps^\alpha}{C_H z^H} 
, \label{eq:CF2}
\end{equation}
which decreases with the range $z$ and with the strength of the random medium, quantified by $d_\alpha$ (see Proposition \ref{prop.1} for the definitions of $H$ and $C_H$).

If the random phase is removed from $\varphi^\eps$ (which means, we observe the field around {the central axis random arrival time}), then the decoherence frequency is larger and it is described in the limit $\eps \to 0$ by the decay in $|\varOmega_1-\varOmega_2|$ of the covariance 
\begin{equation}
\cC (\varOmega_1,\varOmega_2,\bX_1,\bX_2,z ) = 
\EE \left[ \psi  (\varOmega_1,\bX_1,z) \overline{\psi (\varOmega_2,\bX_2,z)}\right].
\label{eq:CF3}
\end{equation}
The evolution equation for this covariance is obtained from the It\^{o}-Schr\"{o}dinger equation in Proposition \ref{prop.2} and the definitions 
\eqref{eq:defgamma}, \eqref{def:Theta},
\begin{align}
&\partial_z \cC (\varOmega_1,\varOmega_2,\bX_1,\bX_2,z ) = \left\{ \frac{i}{2 k_1} \Delta_{\bX_1} - 
\frac{i}{2 k_2} \Delta_{\bX_2} - \left[ \frac{k_1k_2}{4} \Theta(\bX_1-\bX_2)  \right. \right.\nonumber \\ 
&~~\left. \left. + 
 \frac{k_1(k_1-k_2)}{4} \Theta(\bX_1) - \frac{k_2(k_1-k_2)}{4} \Theta(\bX_2) \right] \right\}
  \cC (\varOmega_1,\varOmega_2,\bX_1,\bX_2,z ), \label{eq:CF4}
\end{align}
for $z > 0$, with initial condition 
\begin{equation}
\cC (\varOmega_1,\varOmega_2,\bX_1,\bX_2,z=0) = \hat F(\varOmega_1,\bX_1) \overline{\hat F(\varOmega_2,\bX_2)}. \label{eq:CF5}
\end{equation}
Here we used the notation $k_j = k(\varOmega_j)$, for $j = 1,2$.

The next proposition, proved in Appendix \ref{ap:B}, gives the approximation of $\cC$ 
in the strongly fluctuation  regime  \eqref{eq:SMF}. 
Since we have already described the spatial decorrelation of the wave field in the previous section, we give the approximation at the axis of the beam.

\rc{
\begin{prop}
\label{prop.4}
In the regime (\ref{eq:SMF}), 
the decoherence frequency
\begin{equation}
\varOmega_{\psi} (z;\varOmega) = \frac{2 \varOmega}{ Q(z) R(z)} 
\label{eq:DecfreqPsi}
\end{equation}
is the scale of variation of the covariance  of $\psi$ with respect to the frequency offset $\varOmega_1 - \varOmega_2$ around the frequency $\varOmega$.
More exactly, the covariance evaluated at $\bX_1 = \bX_2 = {\bf 0}$ and at two positive frequencies $\varOmega_1$, $\varOmega_2$ such that $|\varOmega_1-\varOmega_2|
\lesssim \varOmega_{\psi} (z;\varOmega)$ with  $\varOmega=(\varOmega_1+\varOmega_2)/2$ is of the form 
\begin{align}
\cC (\varOmega_1,\varOmega_2,{\bf 0},{\bf 0},z) \simeq  \frac{ {\cF}(\varOmega _1,\varOmega _2) }{R^2(z)} 
\Xi_\alpha  \Big(\frac{\varOmega_1-\varOmega_2 }{ \varOmega_\psi(z;\varOmega)} \Big),
\label{eq:prop4}
\end{align}
with
\begin{equation}
 {\cF}(\varOmega _1,\varOmega _2) = 
\int_{\RR^2} d \obX \, \hat F (\varOmega_1, \obX ) \overline{\hat F (\varOmega_2,\obX  )},
 \label{eq:expressW2IC2f}
\end{equation} 
and  $\Xi_\alpha$ is a function that depends only on $\alpha$. It is defined in equation \eqref{eq:defXi} below 
for dimensionless, $O(1)$ arguments. 
\end{prop}
}

\vspace{0.05in}

\rc{The  decoherence frequency $\varOmega_{\psi} (z;\varOmega) $ given by (\ref{eq:DecfreqPsi}) is proportional to the central frequency $\varOmega$, but it is much smaller  because $Q R \gg 1$ by \eqref{eq:SMF}.}
To define $\Xi_\alpha$, we  introduce the dimensionless and $O(1)$ variables
\begin{equation} 
\tilde \obX = \frac{\obX}{R(z)}, \qquad \tilde \bkappa = \frac{\bkappa}{Q(z)},
\end{equation}
where we anticipate the range dependent radii of spatial and wave vector support of the covariance, using the results in subsection \ref{sect:CovSpatial} and definitions  (\ref{eq:defR}-\ref{eq:defQ}). The range $z$ is fixed here, and we introduce the dimensionless $\tilde z \in [0,1]$, so that $z \tilde z \in [0,z]$.
With this notation we have 
\begin{align}
\label{eq:defXi}
\Xi_\alpha(\tilde{k}) =   \int_{\RR^2} d\tilde{\bkappa}  \, \tilde{\cal \cW}_\alpha(\tilde{k} , \tilde{\bX}={\bf0},\tilde{\bkappa}, \tilde{z}=1)
\end{align}
for dimensionless and $O(1)$ variable $\tilde k$, where $\tilde{\cal \cW}_\alpha$ satisfies
\begin{align}
\nonumber
\left[ \partial_{\tilde z} + \right. &\left. (1+\alpha)^{1/\alpha} \tilde \bkappa \cdot \nabla_{\tilde{\obX}} \right] \tilde \cW_\alpha (\tilde k, \tilde{\obX},\tilde{\bkappa},\tilde z) =
 \frac{2^\alpha \alpha \Gamma(1+\alpha/2)}{8 \pi\Gamma(1-\alpha/2) }
\int_{\RR^2} d \tbq |\tbq|^{-\alpha - 2} \\
&\hspace{1in} \times
 \big[  
 \tilde \cW_\alpha (\tilde k, \tilde{\obX},\tilde{\bkappa}- \tbq,\tilde z) e^{-  i \tilde k \tbq \cdot \tilde{\obX} } 
 -\tilde \cW_\alpha (\tilde k, \tilde{\obX},\tilde{\bkappa},\tilde z)\big], 
  \label{eq:B11}
\end{align}
at $\tilde z > 0$, with initial condition 
\begin{equation}
\tilde \cW_\alpha (\tilde k ,\tilde{\obX},\tilde{\bkappa},\tilde{z}=0) = \delta(\tilde{\obX}) \delta(\tilde \bkappa).
\label{eq:B12}
\end{equation}

By scaling out the range $z$, the beam radius $R$ and the wave vector radius $Q$, we made 
$\tilde \cW_\alpha$ and thus $\Xi_\alpha$ depend only on $\alpha$. Note that when $\tilde k = 0$, which corresponds to 
taking $\varOmega_1 = \varOmega_2 = \varOmega$ in \eqref{eq:prop4}, we recover the result in Proposition 
\ref{prop:SpCov}. Indeed, \eqref{eq:expressW2IC2f} becomes $\hat \cW_{\varOmega,0}({\bf 0},{\bf 0})$, 
per definition \eqref{eq:expressW2IC}, and solving explicitly \eqref{eq:B11} with a calculation similar to that in Appendix \ref{ap:A},  we get 
$$
\tilde \cW_\alpha (\tilde k =0 ,\tilde{\obX},\tilde{\bkappa},\tilde{z}=1) =  \frac{1}{(2\pi)^2} \int_{\RR^2} d \bzeta \, \Phi_\alpha(\tilde{\obX}, \bzeta)
e^{- i \bzeta \cdot \tilde{\bkappa}}  ,
$$ and 
$$
\Xi_\alpha(\tilde{k}=0)= \int_{\RR^2} d \tilde \bkappa\, \tilde \cW_\alpha (0 ,{\bf 0},\tilde{\bkappa},1)  = \int_{\RR^2} d \tilde \bkappa\, \frac{1}{(2\pi)^2} \int_{\RR^2} d \bzeta \, \Phi_\alpha({\bf 0}, \bzeta)
e^{- i \bzeta \cdot \tilde{\bkappa}}  =  \Phi_\alpha({\bf 0},{\bf 0}) .
$$

\subsection{Pulse deformation}
\label{sect:Pulse}
 The wave field {evaluated at the center of the beam and} observed around the {central axis random travel time }
 $z/c_o+\eps^4 {\cal Z}^\eps(z)/c_o$ is  
 \begin{align}
 \nonumber
 U^\eps ( T , z) &= u^\eps\Big( t=\frac{z}{c_o} +\eps^4  \frac{ {\cal Z}^\eps(z)}{c_o} + \eps^4 T, \bx={\bf 0}, z \Big) \\
 &=  \frac{c_o}{4\pi} \int_\RR d \varOmega \, \psi^\eps( \varOmega,{\bf 0},z) e^{-i \varOmega T} .
\end{align}
 In the limit $\eps \to0$ it converges in distribution to
 \begin{equation}
 U(T,z)  = \frac{c_o}{4\pi}\int_{\RR}  d \varOmega \,   \psi(\varOmega,{\bf 0},z) e^{- i \varOmega T},
\end{equation}
where $\psi$ is the solution of the It\^o-Schr\"odinger equation (\ref{eq:itoschro}) with the initial condition~\eqref{eq:itoschroIC}. 

If the source  has the Gaussian spectrum
 $$
 \hat{F} (\varOmega,\bX) = \frac{1}{B}  \Big[
 \exp\Big( - \frac{(\varOmega-\omega_o)^2}{2B^2}\Big)
 +\exp\Big( - \frac{(\varOmega+\omega_o)^2}{2B^2}\Big)
 \Big] S\Big(\frac{\bX}{r_{\rm s}}\Big),
 $$
 then the time-dependent wave field has the form
\begin{align}
U(T,z) =& e^{-i \omega_o T}\widetilde{U}(T,z) + c.c. \label{eq:defUTz}
\end{align}
Here ``c.c." is short notation for the complex conjugate of the first term, 
\begin{align}
\widetilde{U}(T,z) =& \frac{c_o}{4\pi}\int_{\RR}  e^{- i (\varOmega-\omega_o) T}
 \widetilde{\psi}(\varOmega,{\bf 0},z) d \varOmega, 
 \label{eq:tUTz}
\end{align}
and the field $\widetilde{\psi}$ solves the same equation \eqref{eq:itoschro} as $\psi$, but has the 
initial condition
\[
\widetilde{\psi}(\varOmega,\bX,z=0) = \frac{1}{B}  \exp\Big(-\frac{(\varOmega-\omega_o)^2}{2B^2}\Big)S\Big(\frac{\bX}{r_{\rm s}}\Big).
\]

To characterize the pulse profile, let us introduce the mean time-dependent intensity envelope 
\begin{align*}
 I(T,z) & = \EE[|\widetilde{U}(T,z)|^2] \\
 &= \frac{c_o^2}{(4\pi)^2 } \int_{\RR} d \varOmega_1 \int_{\RR} d \varOmega_2 \,  e^{- i(\varOmega_1 - \varOmega_2) T  }  
  \EE \big[ \widetilde{\psi}(\varOmega_1,{\bf 0},z)  \overline{ \widetilde{\psi}(\varOmega_2,{\bf 0},z) } \big]  .
 \end{align*}
In view of Proposition~\ref{prop.4}, 
\rc{if  the condition \eqref{eq:SMF} holds and the bandwidth satisfies $B \lesssim\varOmega_\psi(z;\omega_o)$}, 
then we get 
\begin{align}
I(T,z) 
=& \frac{c_o^2}{{16} \pi^{3/2}} \frac{\int_{\RR^2} d \bX \, |S(\bX/r_{\rm s})|^2  }{B R^2(z) }
\int_{\RR} d \varOmega \, e^{ - \frac{\varOmega^2}{4B^2} - iT \varOmega}\,  \Xi_\alpha\Big( \frac{\varOmega  }{\varOmega_\psi(z;\omega_o) }  \Big).
\end{align}
This result shows that the pulse profile is affected by the random medium via the function $\Xi_\alpha$. For a narrowband pulse, with $B \ll \varOmega_\psi (z,\omega_o)$, the profile is preserved and we have
$$
 I(T,z) 
= \frac{c_o^2}{2\pi} \frac{\int_{\RR^2} d\bX \, |S(\bX/r_{\rm s})|^2  }{R^2(z)} \Phi_\alpha({\bf 0},{\bf 0}) \exp(-B^2 T^2).
 $$
It is only when $B $ is of the same order as $\varOmega_\psi(z,\omega_o)$ that the random medium induces pulse deformation.

\section{Comparison with the results in optics}
\label{sect:Comparison}
In this section we compare the expression of the mean intensity and spectrum  of the wave emerging from the asymptotic paraxial theory in random media and compare it with the results used in the optics literature \cite{andrews2005laser}. Because this literature considers 
time-harmonic waves, we limit the comparison to a fixed frequency $\varOmega$.

We analyzed the spatial covariance $\cC_\varOmega$  in subsection \ref{sect:CovSpatial} for $\alpha \in (0,1)$ and $L_o \to \infty$, where the process $\mu$ has long-range correlations. 
We showed there that the central phase $k(\varOmega) {\cal Z}^\eps$, which is influenced by such correlations, plays no role i.e., $\cC_\varOmega$ is the covariance of $\psi$, the $\eps \to 0$  limit (in distribution) of the wave field $\psi^\eps$ observed in the random travel time frame. Since $\psi^\eps$  experiences the random medium via the mixing process \eqref{eq:defNu}, the results in subsection \ref{sect:CovSpatial}  extend verbatim to the case $\alpha \in (0,1)\cup (1,2)$ and a finite $L_o$ (recall subsection \ref{sect:randmodel}). In particular, the results \eqref{eq:MeanInt}, \eqref{eq:meanSpectr} and \eqref{eq:CovSMF} remain valid as long as 
\begin{equation}
R(z) < L_o, \qquad Q(z) < l_o^{-1}.
\end{equation}

The formulas in \cite{andrews2005laser} are for the Kolmogorov spectrum of turbulence, corresponding to $\alpha = 5/3$.
The radius $R$ of the beam and the spectral radius $Q$ for this $\alpha$ are, from definitions (\ref{eq:defR}-\ref{eq:defQ}),
\begin{equation}
R(z) = \Big(\frac{3}{8} d_{5/3}\Big)^{3/5} z^{8/5} k^{1/5}(\varOmega), \qquad Q(z) = (d_{5/3})^{3/5} z^{3/5} k^{6/5}(\varOmega),
\label{eq:Co1}
\end{equation} 
and $d_{5/3}$ can be written in terms of the normalization constant $\chi_{5/3}$ of the  random process $\mu$ 
using equation \eqref{def:Gamma}
\begin{equation}
d_{5/3} = \frac{3 \Gamma(1/6)}{5 \pi 2^{8/3} \Gamma(11/6)} \chi_{5/3} \approx 0.178 \chi_{5/3}.
\label{eq:Co2}
\end{equation}
To compare with the formulas in \cite{andrews2005laser}, we note that in \cite[Section 3.3.1]{andrews2005laser} the power spectrum of the fluctuations $\tilde \mu$ of the index of refraction is\footnote{The power spectrum is called $\Phi_n$  in  \cite{andrews2005laser}, but 
to avoid confusion with the  function \eqref{eq:defPhiAlpha} we rename it $\mathbb{S}^{\rm A-P}$.}
\begin{equation}
\mathbb{S}^{^{\rm A-P}}\hspace{-0.04in}(\bkappa) = 0.033 C_n^2 |\bkappa|^{-11/3} {\bf 1}_{(L_o^{-1},l_o^{-1})}(|\bkappa|).
\label{eq:Co3}
\end{equation}
Since our process $\mu$ models the fluctuations of the squared index of refraction, we have $\mu \approx 2 \tilde \mu$. 
We also have a different convention of the Fourier transform, which can be reconciled by dividing the formulas in \cite{andrews2005laser} by $(2 \pi)^3$. Then, we obtain from definition \eqref{def:psdS} that our power spectrum $\mathbb{S}$ 
corresponds to \eqref{eq:Co3} at $\alpha = 5/3$, for the normalization constant
$
\chi_{5/3} = 4 (2 \pi)^3 0.033 C_n^2,
$
which gives, from \eqref{eq:Co2},
\begin{equation}
d_{5/3} \approx 5.828 C_n^2.
\end{equation}

\begin{figure}[t]
\begin{center}
\includegraphics[width=6cm]{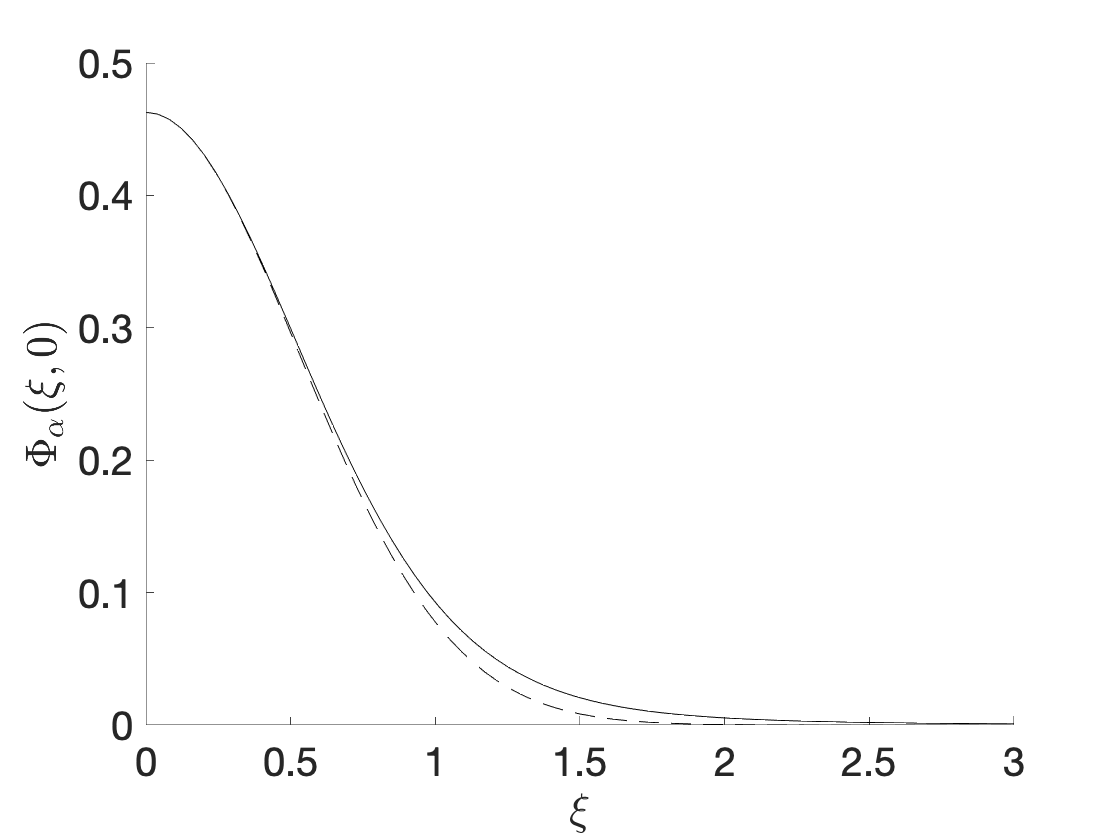}
\includegraphics[width=6cm]{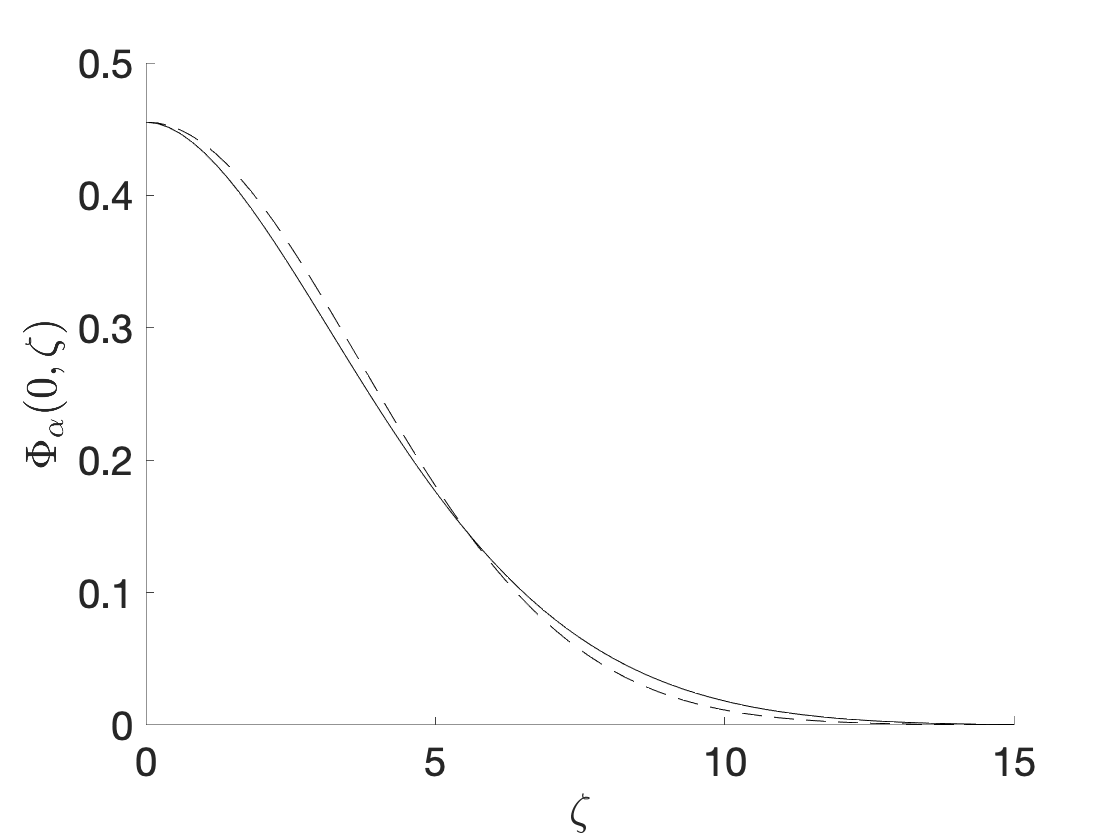} \\
\includegraphics[width=6cm]{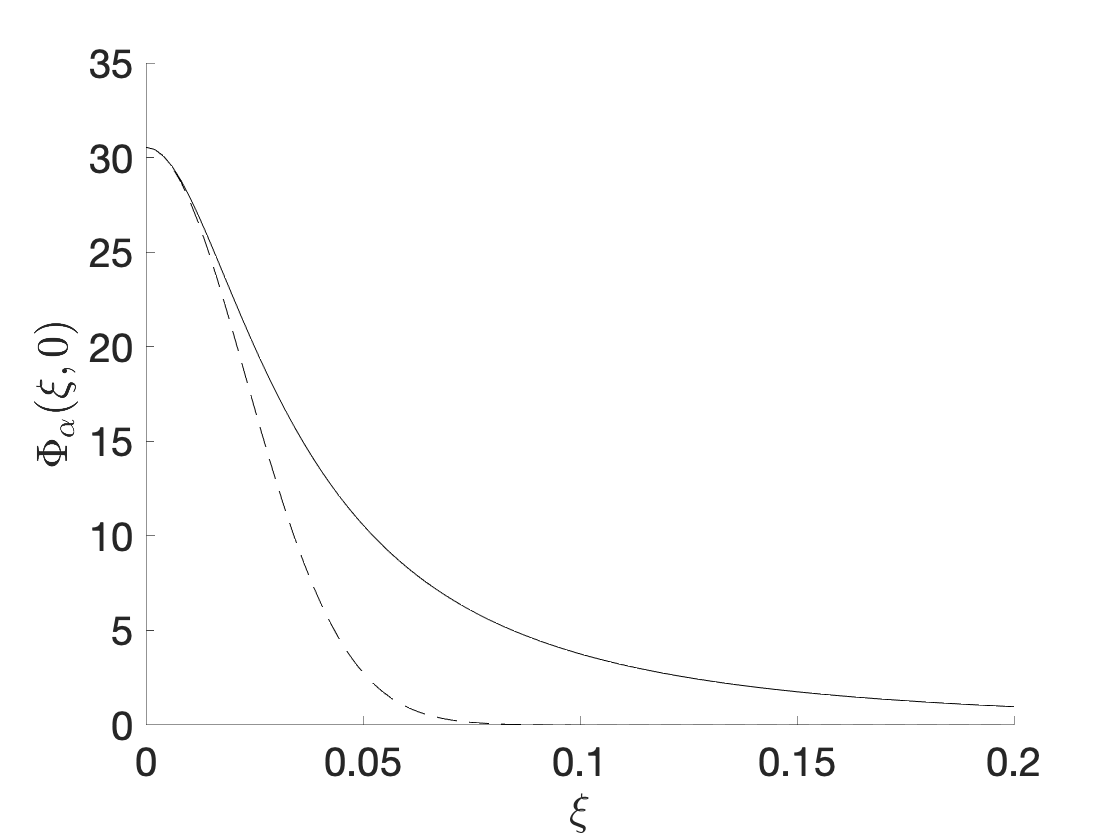}
\includegraphics[width=6cm]{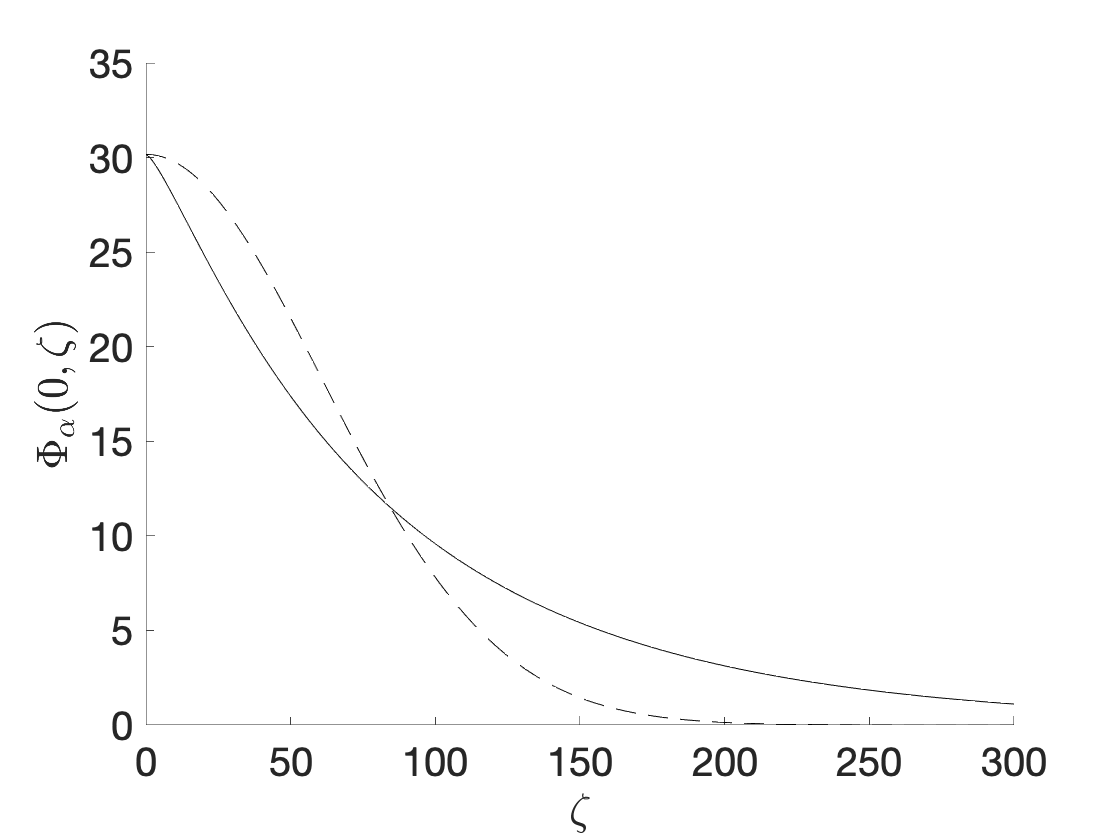}
\end{center}
\caption{
Left: Function $\bxi \mapsto \Phi_\alpha(\bxi,{\bf 0})$ (solid line) and the Gaussian fit $\bxi \mapsto \Phi_\alpha({\bf 0},{\bf 0}) \exp (- q_{\alpha} |\bxi|^2 )$ (dashed line; remember by (\ref{eq:expandpsialpha}) that $\Phi_\alpha(\bxi,{\bf 0})=\Phi_\alpha({\bf 0},{\bf 0}) \big[1-q_{\alpha}|\bxi|^2+o(|\bxi|^2)\big]$). 
Right:  Function $\bzeta \mapsto \Phi_\alpha({\bf 0},\bzeta)$ (solid line)  with the Gaussian fit $\bzeta \mapsto \Phi_\alpha({\bf 0},{\bf 0}) \exp (- |\bzeta|^2 / \zeta_\alpha^2 )$ (dashed line, with $\zeta_{2/3}=86$ and $\zeta_{5/3}=5.2$ determined by least-square fit). 
Top plots: $\alpha=5/3$. Here the Gaussian fits are close to the true  profiles. Bottom plots: $\alpha = 2/3$. Here the Gaussian fits are far from the true  profiles, which have heavy tails.}
\label{fig:Phialpha}
\end{figure}

We begin the comparison with the mean intensity, which is proportional to $\Psi_\alpha(\obX/R)  = \Phi_{\alpha}(\obX/R,{\bf 0})$ per equations \eqref{eq:MeanInt} and \eqref{eq:CovSMF}. This is approximated in \cite[Section 7.3.3]{andrews2005laser}  by a Gaussian
function, which is close to the true profile for $\alpha = 5/3$, as illustrated in the top left plot of Fig. \ref{fig:Phialpha}. In this figure, 
the standard deviation of the Gaussian is $(2 q_{5/3})^{-1/2}$ and $q_{5/3}$ can determined from the expansion (\ref{eq:expandpsialpha}) of $\Psi_{5/3}$ about the origin: $q_{5/3} = \frac{2^{2/5} \Gamma(12/5)}{\Gamma(6/5)} \approx 1.785$.
The radius of the support of the mean intensity defined in \cite[Section 7.3.3]{andrews2005laser}, aka the ``effective spotsize", corresponds to 
\begin{align}
\frac{R(z)}{\sqrt{q_{5/3}}} \approx \frac{ \big(\frac{3}{8} 5.828 C_n^2 \big)^{3/5} }{\sqrt{1.785}} z^{8/5} k^{1/5}(\varOmega) \approx 
1.2 C_n^{6/5} z^{8/5} k^{1/5}(\varOmega), \label{eq:SpotSize}
\end{align}
where we used equations (\ref{eq:Co1}-\ref{eq:Co2}). The effective spotsize is called $W_{\rm LT}$ in \cite{andrews2005laser} and its estimate follows from equations 
(35) and (45) in Section 7.3.3 and the ``Rytov variance" given in Section 7.1. It is given by $1.45 C_n^{6/5} z^{8/5} k^{1/5}$,
which looks like the theoretically derived formula \eqref{eq:SpotSize}, except for the multiplicative constant. Thus, the 
effective spotsize seems to be slightly over-estimated in \cite{andrews2005laser}.

Similarly, we can quantify the ``correlation radius",  which is defined in \cite{andrews2005laser} as the radius of support of the mean spectrum, which is according to equations (\ref{eq:meanSpectr}-\ref{eq:CovSMF}) proportional to $\Psi_{5/3}(\bkappa/Q) = \Phi_{5/3}({\bf 0},\bkappa/Q) $. This  is also  modeled 
as Gaussian in \cite{andrews2005laser}, which is close to the true profile for the  standard deviation $\zeta_{5/3}/\sqrt{2}$, $\zeta_{5/3}\approx 5.2$ (determined by least-square fit), as illustrated in the top right plot of Fig. \ref{fig:Phialpha}. 
The correlation radius is 
\begin{equation}
\frac{\zeta_{5/3}}{Q(z)} \approx \frac{\zeta_{5/3}}{(5.828 C_n^2)^{3/5}} z^{-3/5} k^{-6/5}(\varOmega) \approx 1.81 C_n^{-6/5} z^{-3/5} k^{-6/5}(\varOmega),
\label{eq:correlR}
\end{equation}
where we used equations (\ref{eq:Co1}-\ref{eq:Co2}). This is called $\rho_{\rm pl}$ in \cite[Section 7.3.4]{andrews2005laser} and it is estimated by $1.6 C_n^{-6/5} z^{-3/5} k^{-6/5}(\varOmega)$. Again, we see the similarity with the theoretically derived formula \eqref{eq:correlR}, except for the multiplicative constant that is slightly under-estimated.

Finally, we note that the Gaussian approximations of the mean intensity and spectrum are inadequate for the case $\alpha < 1$, as illustrated in the bottom plots of Fig. \ref{fig:Phialpha}. The theoretically derived formulas \eqref{eq:MeanInt} and \eqref{eq:meanSpectr}
display heavier tails than the best fit Gaussian profiles.

\section{Summary}
\label{sect:Summary}
 
 Kolmogorov's theory for optical  turbulence predicts a  power law 
 form for the spectrum of the fluctuations of the index of refraction.
 In recent years, there has been 
 a shift of focus on non-Kolmogorov turbulence. 
 This is motivated in part  by the analysis of atmospheric temperature recordings
 which show  deviations from the Kolmogorov power spectrum. However,  these studies
 deal mostly with the case of light tails of the two-point statistics for the medium fluctuations, 
 which correspond
 to an integrable covariance function.    Here we consider beam wave propagation
 in  random media with long-range correlations, where the tails  of the 
  covariance function decay at a slower rate  and the medium contains
more features of  low spatial frequency.  
 We explicitly discuss the roles of the inner and 
 outer scales delineating the power law,  and contrast the results with those for the Kolmogorov turbulence. 
 
 A main result in the long-range case is that the randomization of the wave field is multiscale: 
 First, we show  that as the beam wave propagates through the medium, a strong random travel time
perturbation builds up. We present a precise characterization of the travel time perturbation, 
which corresponds to  a fractional Brownian motion, with Hurst index and amplitude determined 
by the statistics of the medium. Second, we show that if we 
observe the beam wave at large propagation distances where the travel time correction is  large
relative  to the pulse width, then the beam wave pulse shape itself is deformed and becomes
random due to scattering. 

Another important result is  a detailed characterization
of the decorrelation of the random beam wave both in space and frequency. This is carried out 
in the random travel time centered frame because otherwise, the  frequency decorrelation would be masked by the very large random phase associated with the travel time fluctuations.  
The analysis reveals a cusp like behavior for the spatial correlations of the wave field 
in the transverse coordinates, with the cusp shape depending on the rate of decay 
of the covariance of the medium fluctuations.   The scale of frequency decorrelation is also quantified and 
it is used to analyze the deformation of the probing pulse induced by scattering.

The results of our analysis  are important
for applications like imaging and communication through the atmosphere,
and also for propagation through the earth's crust or through the ocean.
 In the case of communication applications a characterization of the statistics of 
fading or strong pulse deformation is important in order to evaluate the efficiency of various 
communication  protocols.  
In imaging through complex media,  one needs  to take into account the geometric 
wavefront distortion that is caused by the random travel time as well as the deformation or blurring of 
the beam pulse shape. 
Quantitive insights about  these effects are useful when
designing schemes for clutter and turbulence compensation.  

\section*{Acknowledgements}
This material is based upon work supported by the Air Force Office of Scientific Research under award numbers FA9550-22-1-0077 and  FA9550-22-1-0176, 
by the U.S. Office of Naval Research under award number N00014-21-1-2370 and
by  the National Science Foundation under grant DMS-2010046.

\appendix
\section{Proof of Proposition \ref{prop.3}}
\label{ap:A}
Equation \eqref{eq:SPC1} written in the coordinates  \eqref{eq:coordC} is 
\begin{equation}
\partial_z \cC_\varOmega\Big(\obX + \frac{\tbX}{2},\obX - \frac{\tbX}{2},z \Big) = \Big[ \frac{i}{k(\varOmega)} \nabla_\obX \cdot \nabla_{\tbX} - \frac{k^2(\varOmega)}{4} 
\Theta(\tbX) \Big]\cC_\varOmega\Big(\obX + \frac{\tbX}{2},\obX - \frac{\tbX}{2},z \Big),
\label{eq:A1}
\end{equation}
and using the Fourier transform 
\begin{equation}
\hat \cW_\varOmega(\bq,\tbX,z) = \int_{\RR^2} d \obX \,  \cC_\varOmega\Big(\obX + \frac{\tbX}{2},\obX - \frac{\tbX}{2},z \Big) e^{- i \bq \cdot \obX},
\label{eq:A2}
\end{equation}
we get 
\begin{align}
\Big( \partial_z + \frac{\bq}{k(\varOmega)} \cdot \nabla_{\tbX} \Big) \hat \cW_\varOmega(\bq,\tbX,z) = - \frac{k^2(\varOmega)}{4} 
\Theta(\tbX) \hat \cW_\varOmega(\bq,\tbX,z), \label{eq:A3}
\end{align}
for $z > 0$, with initial condition $\hat \cW_\varOmega(\bq,\tbX,0) = \hat \cW_{\varOmega,0}(\bq,\tbX), $ defined in \eqref{eq:expressW2IC}. 

We can solve \eqref{eq:A3} by integration along the characteristic $\tbX = \tbX_0 + \bq z/k(\varOmega)$, starting from $\tbX_0$, using that 
\begin{align*}
\partial_z \hat \cW_\varOmega\Big(\bq,\tbX_0 + \frac{\bq}{k(\varOmega)}z,z\Big) &= \Big( \partial_z + \frac{\bq}{k(\varOmega)} \cdot \nabla_{\tbX} \Big)\cW_\varOmega\Big(\bq,\tbX_0 + \frac{\bq}{k(\varOmega)}z,z\Big) \\& = -  \frac{k^2(\varOmega)}{4} 
\Theta(\tbX)\cW_\varOmega\Big(\bq,\tbX_0 + \frac{\bq}{k(\varOmega)}z,z\Big), \quad z > 0.
\end{align*}
The result is 
\begin{align*}
\hat \cW_\varOmega\Big(\bq,\tbX_0 + \frac{\bq}{k(\varOmega)}z,z\Big) = \hat \cW_{\varOmega,0}(\bq,\tbX_0) \exp \left[-\frac{k^2(\varOmega)}{4} \int_0^z dz' \, \Theta \Big(\tbX_0 + 
\frac{\bq}{k(\varOmega)}z' \Big) \right],
\end{align*}
or, equivalently, in terms of $\tbX$, 
\begin{align*}
\hat \cW_\varOmega\Big(\bq,\tbX,z\Big) = \hat \cW_{\varOmega,0}\Big(\bq,\tbX -\frac{\bq}{k(\varOmega)}z \Big) \exp \left[-\frac{k^2(\varOmega)}{4} \int_0^z dz' \, \Theta \Big(\tbX - 
\frac{\bq}{k(\varOmega)}(z-z') \Big) \right].
\end{align*}
The result stated in Proposition \ref{prop.3} follows from this expression and the definition \eqref{eq:SPC3} of the Wigner transform,
\begin{align*}
\cW_\varOmega(\obX,\bkappa,z) &= \int_{\RR^2} d \tbX \, \cC_\varOmega\Big(\obX+\frac{\tbX}{2},\obX-\frac{\tbX}{2},z\Big) \exp(-i \bkappa \cdot \tbX) \\
&=
\int_{\RR^2} d \tbX \,  \int_{\RR^2} \frac{d \bq}{(2 \pi)^2} \hat \cW_\varOmega\big(\bq,\tbX,z\big) \exp \big(i \bq \cdot \obX - i \bkappa \cdot \tbX\big) \\
&= \frac{1}{(2 \pi)^2} \int_{\RR^2} d \bq \int_{\RR^2} d \tbX \, \, \hat  \cW_{\varOmega,0}\Big(\bq,\tbX -\frac{\bq}{k(\varOmega)}z \Big) \exp \big(i \bq \cdot \obX - i \bkappa \cdot \tbX\big) \\
& \qquad \qquad \qquad \times \exp\left[-\frac{k^2(\varOmega)}{4} \int_0^z dz' \, \Theta \Big(\tbX -
\frac{\bq}{k(\varOmega)}(z-z') \Big)\right].
\end{align*}
In \eqref{eq:expressW2} we used the change of variable  $\tbX' = \tbX -\frac{\bq}{k(\varOmega)}z $. $~~\Box$

\section{Proof of the expansion (\ref{eq:expandPhizeta})}
\label{app:C}
We first remark that
\[
|\by|^\alpha-|\bx|^\alpha = \mathfrak{C}_\alpha
\int_{\RR^2} d\bq  |\bq|^{-\alpha-2} \big( e^{i \bq\cdot\bx}-e^{i \bq \cdot \by}\big) , \]
with constant $\mathfrak{C}_\alpha$ defined by 
\[\mathfrak{C}_\alpha^{-1} = 2\pi \int_0^\infty \big(1-J_0(s) \big) s^{-1-\alpha} ds.\]
Next, we compute from (\ref{eq:defPhiAlpha}):
\begin{align}
\Phi_\alpha({\bf 0},\bzeta) -
\Phi_\alpha({\bf 0},{\bf 0}) 
= -
\Phi_{\alpha,1}(\bzeta) (1+o(1)), \label{eq:Phi1}
\end{align}
with
\begin{align*}
\Phi_{\alpha,1}(\bzeta) &=
\frac{1}{4(2\pi)^2} \int_{\RR^2} d\bu
e^{-\frac{1}{4}|\bu|^\alpha}
\int_0^1 ds \big( 
|\bzeta-(1+\alpha)^{1/\alpha} \bu s|^{\alpha}
-
|(1+\alpha)^{1/\alpha} \bu s|^{\alpha} \big)\\
&= \frac{\mathfrak{C}_\alpha}{4(2\pi)^2} \int_{\RR^2} d\bu\, 
e^{-\frac{1}{4}|\bu|^\alpha}
\int_0^1 ds \int_{\RR^2}d\bq \, 
|\bq|^{-\alpha-2} 
e^{i (1+\alpha)^{1/\alpha} \bu \cdot\bq s}
\big(1-e^{- i \bzeta\cdot\bq}\big) \\
&= \frac{\mathfrak{C}_\alpha}{4} \int_0^\infty d\eta \, 
\eta e^{-\frac{1}{4}\eta^\alpha}
\int_0^1 ds \int_0^\infty dq \, 
q^{-\alpha-1} 
J_0\big( (1+\alpha)^{1/\alpha} \eta q s \big)
\big(1-J_0(|\bzeta|q) \big) \\
&= \frac{\mathfrak{C}_\alpha}{4  (1+\alpha)^{1/\alpha} } \int_0^\infty d\eta
e^{-\frac{1}{4}\eta^\alpha}
\int_0^\infty dq \, 
q^{-\alpha-2} 
{\cal J}_0\big( (1+\alpha)^{1/\alpha} \eta q \big)
\big(1-J_0(|\bzeta| q) \big)  ,
\end{align*}
where ${\cal J}_o(s)=\int_0^s J_0(s') ds'$ is the antiderivative of the Bessel function $J_0$.  It is a bounded function that converges to one as $s\to+\infty$.
By the change of variable $s=|\bzeta| q$, we get
\begin{align*}
\Phi_{\alpha,1}(\bzeta)
= \frac{C_\alpha |\bzeta|^{\alpha+1} }{4  (1+\alpha)^{1/\alpha} } \int_0^\infty d\eta
e^{-\frac{1}{4}\eta^\alpha}
 \int_0^\infty ds\, 
s^{-\alpha-2} 
{\cal J}_0\left( \frac{(1+\alpha)^{1/\alpha} \eta s}{|\bzeta|} \right)
\big(1-J_0(s) \big) .
\end{align*}
Using the dominated convergence theorem, we find
$$
\frac{\Phi_{\alpha,1}(\bzeta) }{|\bzeta|^{\alpha+1}} \stackrel{|\bzeta|\to0}{\longrightarrow}
 \frac{C_\alpha  }{4  (1+\alpha)^{1/\alpha} } \int_0^\infty d\eta \, 
e^{-\frac{1}{4}\eta^\alpha}
 \int_0^\infty ds \, 
s^{-\alpha-2} 
\big(1-J_0(s) \big).
$$
Therefore, equation \eqref{eq:Phi1} gives the expansion
$$
\Phi_\alpha({\bf 0},\bzeta) =
\Phi_\alpha({\bf 0},{\bf 0}) 
\big( 1 -r_\alpha |\bzeta|^{\alpha+1}+o(|\bzeta|^{\alpha+1})\big),
$$
with
$$
r_\alpha=   \frac{  \int_0^\infty d\eta \, 
e^{-\frac{1}{4}\eta^\alpha} 
 \int_0^\infty ds \, 
s^{-\alpha-2} 
\big(1-J_0(s) \big)ds}{8\pi   (1+\alpha)^{1/\alpha} \Phi_\alpha({\bf 0},{\bf 0}) \int_0^\infty s^{-\alpha-1}  \big(1-J_0(s) \big)  }.
$$
The desired result follows once we use
$$
\Phi_\alpha({\bf 0},{\bf 0}) = \frac{1}{2\pi} \int_0^\infty d \eta \, \eta e^{-\frac{1}{4}\eta^\alpha}   ,
$$
and the identities
\begin{align*}
\int_0^\infty ds \, s^{-\alpha-1} \big(1-J_0(s) \big)  &= 
\frac{2^{-\alpha}}{\alpha} \frac{\Gamma(1-\alpha/2)}{\Gamma(1+\alpha/2)}, \\
 \int_0^\infty ds \, 
s^{-\alpha-2} 
\big(1-J_0(s) \big) 
&=\frac{2^{-\alpha-1}}{\alpha+1}  \frac{\Gamma(1/2-\alpha/2)}{\Gamma(3/2+\alpha/2)}, \\
\int_0^\infty  d \eta \, \eta
e^{-\frac{1}{4}\eta^\alpha}
&= \frac{2^{4/\alpha}}{\alpha} \Gamma(\frac{2}{\alpha}), \\
\int_0^\infty d\eta \, 
e^{-\frac{1}{4}\eta^\alpha} 
&=  \frac{2^{2/\alpha}}{\alpha} \Gamma(\frac{1}{\alpha}).
\end{align*}

\section{Proof of Proposition \ref{prop.4}}
\label{ap:B}
Let us  introduce the reference wavenumber $k$ and
use it to change coordinates in the cross-range plane, as follows
\begin{equation}
\bX_1 = \sqrt{\frac{k}{k_1}} \Big(\obX + \frac{\tbX}{2}\Big), \qquad 
\bX_2 = \sqrt{\frac{k}{k_2}} \Big(\obX - \frac{\tbX}{2}\Big).
\label{eq:B1}
\end{equation}
Writing the evolution equation \eqref{eq:CF4} in these coordinates and then taking the Fourier transform in $\tbX$, which defines  the  Wigner transform 
\begin{equation*}
\cW(\varOmega_1,\varOmega_2,\obX,\bkappa,z) = \int_{\RR^2} d \tbX \, 
 \cC \Big(\varOmega_1,\varOmega_2, \sqrt{\frac{k}{k_1}} \Big(\obX + \frac{\tbX}{2}\Big),\sqrt{\frac{k}{k_2}} \Big(\obX - \frac{\tbX}{2}\Big),z \Big) e^{-i \bkappa \cdot \tbX} ,
\end{equation*}
we obtain the following equation 
\begin{align}
    &\Big( \partial_z  +\frac{1}{k} \bkappa \cdot \nabla_\obX \Big) {\cal W}(\varOmega_1,\varOmega_2,\obX,\bkappa,z) = -
    \frac{1}{4(2\pi)^2}\int_{\RR^2} d\bq \, \hat{\Theta}(\bq) \nonumber 
    \\& \quad 
    \times \Big\{
    k_1 k_2 {\cal W}\Big(\varOmega_1,\varOmega_2,\obX,\bkappa - \frac{\bq}{2} \Big( \sqrt{\frac{k}{k_1}}+\sqrt{\frac{k}{k_2}}\Big),z
    \Big) e^{i \obX \cdot \bq  \left(\sqrt{\frac{k}{k_1}}-\sqrt{\frac{k}{k_2}}\right)} \nonumber 
       \\& \quad 
    \quad +k_1(k_1-k_2) {\cal W}\Big(\varOmega_1,\varOmega_2,\obX,\bkappa - \frac{\bq}{2} \sqrt{\frac{k}{k_1}},z
    \Big) e^{i \obX \cdot \bq  \sqrt{\frac{k}{k_1}}} \nonumber \\
    &\quad \quad -
     k_2(k_1-k_2) {\cal W}\Big(\varOmega_1,\varOmega_2,\obX,\bkappa + \frac{\bq}{2} \sqrt{\frac{k}{k_2}},z
    \Big) e^{i \obX \cdot \bq  \sqrt{\frac{k}{k_2}}}
    \Big\} , \label{eq:B}
\end{align}
for $z >0$, where the net effect of the random medium is in the Fourier transform $\hat \Theta$ of the function $\Theta$ defined in (\ref{def:Theta}).

Although we are interested in an infinite outer scale, let us consider a modification of \eqref{def:Theta}, corresponding to a finite $L_o$, 
\begin{equation*}
\Theta_{_{L_o}} (\bX) = \frac{\chi_\alpha}{2 \pi} \int_{L_o^{-1}}^{l_o^{-1}} d \kappa \, \left[ 1- J_0(\kappa |\bX|) \right] \kappa^{-1 - \alpha} = 
\Theta(\bX) + O\Big(\frac{\chi_\alpha |\bX|^2}{ L_o^{2-\alpha}}\Big) \stackrel{L_o \to \infty}{\longrightarrow} \Theta(\bX).
\end{equation*}
The  Fourier transform of this function is 
\begin{align*}
\hat \Theta_{_{L_o}}(\bq) &= \int_{\RR^2} d \bX \, \Theta_{_{L_o}} (\bX) e^{-i \bq \cdot \bX} = 2 \pi \chi_\alpha \frac{(L_o^{\alpha} - l_o^\alpha)}{\alpha} \delta (\bq) - \chi_\alpha |\bq|^{-2 - \alpha}{\bf 1}_{(L_o^{-1},l_o^{-1})}(|\bq|)
\end{align*}
and  we explain next that   equation  \eqref{eq:B} makes sense for $L_o \to \infty$.  Using the observation 
\[
\int_{\RR^2} d \bq \, {\bf 1}_{(L_o^{-1},l_o^{-1})}(|\bq|) |\bq|^{-2 - \alpha} = 2 \pi \int_0^\infty dq \, {\bf 1}_{(L_o^{-1},l_o^{-1})}(q) q^{-1 - \alpha} = \frac{2 \pi (L_o^\alpha - l_o^\alpha)}{\alpha},
\]
we can rewrite \eqref{eq:B}, with $\hat \Theta$ replaced by $\hat \Theta_{_{L_o}}$ and therefore $\cW$ replaced by $\cW_{_{L_o}}$ as follows
\begin{align}
    &\hspace{-0.04in}\Big( \partial_z  +\frac{1}{k} \bkappa \cdot \nabla_\obX\Big) {\cal W}_{_{L_o}}(\varOmega_1,\varOmega_2,\obX,\bkappa,z) = 
    \frac{\chi_\alpha}{4(2\pi)^2}\int_{\RR^2} d\bq \, {\bf 1}_{(L_o^{-1},l_o^{-1})}(|\bq|) |\bq|^{-2-\alpha}  \Big\{
    k_1 k_2  \nonumber 
    \\& 
    \times \Big[\cW_{_{L_o}}\Big(\varOmega_1,\varOmega_2,\obX,\bkappa - \frac{\bq}{2} \Big( \sqrt{\frac{k}{k_1}}+\sqrt{\frac{k}{k_2}}\Big),z
    \Big) e^{i \obX \cdot \bq  \left(\sqrt{\frac{k}{k_1}}-\sqrt{\frac{k}{k_2}}\right)} - \cW_{_{L_o}}(\varOmega_1,\varOmega_2,\obX,\bkappa,z)\Big]\nonumber 
       \\& 
    +k_1(k_1-k_2) \Big[{\cal W}_{_{L_o}}\Big(\varOmega_1,\varOmega_2,\obX,\bkappa - \frac{\bq}{2} \sqrt{\frac{k}{k_1}},z
    \Big) e^{i \obX \cdot \bq  \sqrt{\frac{k}{k_1}}} - \cW_{_{L_o}}(\varOmega_1,\varOmega_2,\obX,\bkappa,z)\Big]\nonumber \\
    &
    -
     k_2(k_1-k_2) \Big[{\cal W}_{_{L_o}}\Big(\varOmega_1,\varOmega_2,\obX,\bkappa + \frac{\bq}{2} \sqrt{\frac{k}{k_2}},z
    \Big) e^{i \obX \cdot \bq  \sqrt{\frac{k}{k_2}}} - \cW_{_{L_o}}(\varOmega_1,\varOmega_2,\obX,\bkappa,z)\Big]
    \Big\} . \label{eq:B3}
    \end{align}
At $|\bq| \sim L_{o}^{-1} \to 0$ the square brackets in this expression are $O(|\bq|)$, and after writing the 
$\bq$ integral in polar coordinates we conclude that the integrand is $O(|\bq|^{-\alpha})$. Thus, after the integration in $|\bq|$  the right-hand side depends on the outer scale as $L_o^{-(1-\alpha)}$. This vanishes as $L_o \to \infty$, so we can take the limit in~\eqref{eq:B3}  and replace $\cW_{_{L_o}}$ by $\cW$. 

Since the integrand in \eqref{eq:B3}  has a fast decay at $|\bq| \to \infty$, like  $|\bq|^{-1-\alpha}$, and we are interested in a small inner scale 
(recall section \ref{sect:CovSpatial}), we can approximate $\cW$ by taking the limit $l_o \to 0$. We obtain the equation \begin{align}
    &\Big( \partial_z  +\frac{1}{k} \bkappa \cdot \nabla_\obX\Big) {\cal W}(\varOmega_1,\varOmega_2,\obX,\bkappa,z) = 
    \frac{\chi_\alpha}{4(2 \pi)^2}\int_{\RR^2} d\bq |\bq|^{-2-\alpha}  \Big\{
    k_1 k_2  \nonumber 
    \\& ~ ~
    \times \Big[\cW\Big(\varOmega_1,\varOmega_2,\obX,\bkappa - \frac{\bq}{2} \Big( \sqrt{\frac{k}{k_1}}+\sqrt{\frac{k}{k_2}}\Big),z
    \Big) e^{i \obX \cdot \bq  \left(\sqrt{\frac{k}{k_1}}-\sqrt{\frac{k}{k_2}}\right)} - \cW(\varOmega_1,\varOmega_2,\obX,\bkappa,z)\Big]\nonumber 
       \\& 
    ~~ +k_1(k_1-k_2) \Big[{\cal W}\Big(\varOmega_1,\varOmega_2,\obX,\bkappa - \frac{\bq}{2} \sqrt{\frac{k}{k_1}},z
    \Big) e^{i \obX \cdot \bq  \sqrt{\frac{k}{k_1}}} - \cW(\varOmega_1,\varOmega_2,\obX,\bkappa,z)\Big]\nonumber \\
    &~~-
     k_2(k_1-k_2) \Big[{\cal W}\Big(\varOmega_1,\varOmega_2,\obX,\bkappa + \frac{\bq}{2} \sqrt{\frac{k}{k_2}},z
    \Big) e^{i \obX \cdot \bq  \sqrt{\frac{k}{k_2}}} - \cW(\varOmega_1,\varOmega_2,\obX,\bkappa,z)\Big]
    \Big\} ,
\end{align}
for $z > 0$, with  the initial condition
\begin{equation}
\label{eq:inicW}
\cW(\varOmega_1,\varOmega_2, \obX,\bkappa,0) = \int_{\RR^2} d \tbX \, \hat F \Big(\varOmega_1, \sqrt{\frac{k}{k_1}} \big(\obX+ \frac{\tbX}{2}\big) \Big) 
\overline{\hat F \Big(\varOmega_2, \sqrt{\frac{k}{k_2}} \big(\obX-\frac{\tbX}{2}\big) \Big) } e^{- i \bkappa \cdot \tbX}.
\end{equation}

Let us consider a range $Z$ in the strong fluctuation medium (\ref{eq:SMF})
so that we have $ Q_Z R_Z \gg 1$ with
\begin{equation}
Q_Z = Q(Z) =(d_\alpha k^2 Z)^{1/\alpha}, \qquad R_Z = R(Z) = \left(\frac{d_\alpha k^{2-\alpha} Z^{\alpha+1}}{\alpha+1}\right)^{1/\alpha}.
\label{eq:defQZRZ}
\end{equation}
 We now show that the decoherence frequency i.e., the scale of decay
  of $\cW$ with respect to $|\varOmega_1-\varOmega_2|$  is $c_o K_Z$, where
 \begin{equation}
K_Z = \frac{2 k}{Q_Z R_Z} \ll k.
\label{eq:B5}
\end{equation}
Indeed, suppose that 
\begin{equation}
k_j = k(\varOmega_j) = k + K_Z \tilde k_j, \qquad j = 1,2,
\label{eq:B6}
\end{equation}
where $\tilde k_j$ are dimensionless $O(1)$ scaled wavenumber offsets with respect to $k$. Then,
\[
\frac{(k_1+k_2)}{2} \simeq k \Big[1  + O\Big(\frac{1}{Q_Z R_Z}\Big) \Big],
\]
and 
\[
k_1 - k_2 = K_Z (\tilde k_1 - \tilde k_2) = O(K_Z)  \ll k.
\]
Introduce also the dimensionless variables 
\begin{equation}
\tilde{\obX} = \frac{\obX}{R_Z}, \quad \tilde{\bq} = \frac{\bq}{Q_Z}, \quad \tilde \bkappa = \frac{\bkappa}{Q_Z}, \quad \tilde z = \frac{z}{Z}.
\label{eq:B7}
\end{equation}
Then, the Wigner transform can be approximated by 
\begin{equation}
\cW(\varOmega_1,\varOmega_2, \obX,\bkappa,z) \approx \frac{(2 \pi)^2 
  \cF(\varOmega_1,\varOmega_2)  }{(R_Z Q_Z)^2} 
\tilde \cW (\tilde k_1-\tilde k_2, \tilde{\obX},\tilde{\bkappa},\tilde z),
\label{eq:B8}
\end{equation}
with  $\cF$ defined in \eqref{eq:expressW2IC2f} and the function $\tilde \cW$ of dimensionless $O(1)$ arguments satisfying equation 
(\ref{eq:B11}), with  initial condition 
(\ref{eq:B12}).

To derive \eqref{eq:B11} we used definitions \eqref{eq:defQZRZ} and \eqref{def:Gamma}, which give 
\begin{equation}
\partial_z + \frac{1}{k} \bkappa \cdot \nabla_{\obX} = \frac{1}{Z} \Big( \partial_{\tilde z} + \frac{Z Q_Z}{k R_Z}  \tilde \bkappa \cdot \nabla_{\tilde{\obX}} \Big)
=  \frac{1}{Z} \left( \partial_{\tilde z} +  (1+\alpha)^{1/\alpha} \tilde \bkappa \cdot \nabla_{\tilde{\obX}} \right),
\label{eq:B15}
\end{equation}
and 
\begin{equation}
Z \chi_\alpha k^2 Q_Z^{-\alpha} = \frac{\chi_\alpha}{d_\alpha} = \frac{2^{\alpha+1} \pi  \alpha \Gamma(1+\alpha/2)}{\Gamma(1-\alpha/2)}. 
\label{eq:B16}
\end{equation}
We also used  (\ref{eq:B5}-\ref{eq:B6}) and neglected the small, $O\big((\tilde k_1-\tilde k_2)K_Z / k\big)$ residual.

To justify the initial condition \eqref{eq:B12} we note first 
that 
in the regime defined by  (\ref{eq:B5}-\ref{eq:B6}) 
we have 
\begin{align}
\nonumber 
 \cW(\varOmega_1,\varOmega_2,\obX,\bkappa,0)    \approx &   \frac{r_{\rm s}^2}{B^2} 
\tilde{\cW}_{\rm s} \Big(\frac{\tilde{\obX}}{r_{\rm s}/R_Z}, \frac{\tilde \bkappa}{1/(r_{\rm s} Q_Z)} \Big)  \left[
  \hat f\Big(\frac{\varOmega_1-\omega_o}{B} \Big) + \hat f\Big(\frac{\varOmega_1+\omega_o}{B} \Big)   \right] \\ 
  &  
\times 
 \left[ \overline{ \hat f\Big(\frac{\varOmega_2-\omega_o}{B} \Big) + \hat f\Big(\frac{\varOmega_2+\omega_o}{B} \Big) }  \right] 
,  
\label{eq:B13}
\end{align} 
where 
\begin{align}
\tilde{\cW}_{\rm s} (\tilde{\obX} , \tilde \bkappa ) = \int_{\RR^2}  d \bxi \, S\big(\tilde{\obX} + \frac{\bxi}{2} \big) \overline{S\big(  \tilde{\obX} - \frac{\bxi}{2} \big) }
e^{- i  \tilde \bkappa \cdot \bxi} \label{eq:B14}
\end{align}
is the dimensionless Wigner transform of the source function $S$. Since $r_{\rm s}/R_Z \ll 1$ and $1/(r_{\rm s} Q_Z) \ll 1$ by \eqref{eq:SMF} and \eqref{eq:defQZRZ},
we conclude from (\ref{eq:B13}-\ref{eq:B14})  that the initial condition is supported at $\tilde \bX \approx {\bf 0}$ and $\tilde \bkappa \approx {\bf 0}$.
This is why we use the Dirac delta in \eqref{eq:B12}. The normalization in \eqref{eq:B8} comes from the identity
\begin{align}
& \int_{\RR^2} d \tilde{\obX} \int_{\RR^2} d \tilde \bkappa \, \cW(\varOmega_1,\varOmega_2, R_Z \tilde{\obX}, Q_Z \tilde \bkappa,0) 
 = \frac{(2 \pi)^2}{(R_Z Q_Z)^2}  \cF(\varOmega_1,\varOmega_2),
\end{align} 
derived from (\ref{eq:B13}-\ref{eq:B14}), with $\cF$ defined in \eqref{eq:expressW2IC2f}. 


\bibliography{biblio} \bibliographystyle{siam}
\end{document}